%% file: ArxivCorruptedEDMcompletion.tex
\documentclass[11pt]{article}
\usepackage{cite}

\usepackage{multirow}
\usepackage{kbordermatrix}

\usepackage{multicol}
\usepackage{adjustbox}

\usepackage{array}
\newcolumntype{P}[1]{>{\centering\arraybackslash}p{#1}}
\usepackage{longtable}

\usepackage{subfigure}
\usepackage{graphicx, color, graphpap}
\usepackage{amsmath,amssymb,amsthm}
\usepackage{ifthen}
\usepackage{bbm}
\usepackage{mathrsfs}
\usepackage{mathtools}
\usepackage{enumitem}
\usepackage{physics}
\usepackage{pstricks}
\usepackage{lineno}
\usepackage{float}
\usepackage{algorithm,algorithmicx,algpseudocode}
\usepackage{fancyhdr}
\usepackage[us,12hr]{datetime} 
\usepackage{exscale}
\usepackage{tabularx}
\usepackage{latexsym}
\usepackage{fullpage}       
\usepackage{float}
\usepackage{verbatim}
\usepackage{multirow}
\usepackage{overpic}
\usepackage{comment} 
\usepackage{lscape}
\usepackage{booktabs}
\usepackage{makeidx}
\usepackage{mathrsfs}

\numberwithin{equation}{section}  
\numberwithin{table}{section}
\numberwithin{figure}{section}
\numberwithin{algorithm}{section}
\makeindex

\usepackage{csvsimple}

\usepackage[pagebackref=true]{hyperref}
\hypersetup{
plainpages=false,       
unicode=false,          
pdftoolbar=true,        
pdfmenubar=true,        
pdffitwindow=false,     
pdfstartview={FitH},    
pdftitle={Single Element Error Correction in EDM
},    
pdfnewwindow=true,      
colorlinks=true,        
linkcolor=blue,         
citecolor=green,        
filecolor=magenta,      
urlcolor=cyan           
}

\usepackage[noabbrev]{cleveref}

\DeclareMathOperator{\argmin}{argmin}

\usepackage{mathtools}
\DeclarePairedDelimiter{\ceil}{\lceil}{\rceil}

\usepackage{refcount} 

\def\R{\mathbb{R}}

\def\Sc{\mathbb{S}}
\def\Sn{\Sc^n}

\def\Snp{\Sc_+^n}

\def\Snpp{\Sc_{++}^n}

\def\Rk{\mathbb{R}^k}

\def\Rd{\mathbb{R}^d}
\def\Rn{\mathbb{R}^n}

\def\Rnp{\mathbb{R}_+^n}

\def\eqref#1{{\normalfont(\ref{#1})}}

\def\eqref#1{{\normalfont(\ref{#1})}}

\newtheorem{theorem}{Theorem}[section]

\newtheorem{definition}[theorem]{Definition}
\newtheorem{example}[theorem]{Example}

\newtheorem{lem}[theorem]{Lemma}
\newtheorem{thm}[theorem]{Theorem}
\newtheorem{cor}[theorem]{Corollary}

\newtheorem{remark}[theorem]{Remark}

\newtheorem{problem}[theorem]{Problem}

\newtheorem{lemma}[theorem]{Lemma}

\crefname{theorem}{Theorem}{Theorems}
\Crefname{theorem}{Theorem}{Theorems}
\crefname{thm}{Theorem}{Theorems}
\Crefname{thm}{Theorem}{Theorems}
\crefname{problem}{Problem}{Theorems}
\Crefname{problem}{Problem}{Theorems}
\Crefname{assump}{Assumption}{Theorems}
\crefname{assump}{Assumption}{Theorems}
\crefname{conjecture}{Conjecture}{Theorems}
\Crefname{conjecture}{Conjecture}{Theorems}
\crefname{prop}{Proposition}{Propositions}
\Crefname{prop}{Proposition}{Propositions}
\crefname{cor}{Corollary}{Corollaries}
\Crefname{cor}{Corollary}{Corollaries}
\Crefname{corollary}{Corollary}{Corollaries}
\Crefname{coroll}{Corollary}{Corollaries}
\crefname{lem}{Lemma}{Lemmas}
\Crefname{lem}{Lemma}{Lemmas}
\theoremstyle{definition}
\crefname{defn}{definition}{definitions}
\Crefname{defn}{Definition}{Definitions}
\crefname{conj}{Conjecture}{Conjectures}
\Crefname{conj}{Conjecture}{Conjectures}
\crefname{remark}{Remark}{Remarks}
\Crefname{remark}{Remark}{Remarks}
\crefname{rmk}{Remark}{Remarks}
\Crefname{rmk}{Remark}{Remarks}
\crefname{example}{Example}{Examples}
\Crefname{example}{Example}{Examples}
\Crefname{case}{Case}{Cases}
\Crefname{Case}{Case}{Cases}
\crefname{align}{}{}
\Crefname{align}{}{}
\crefname{equation}{}{}
\Crefname{equation}{}{}

\newcommand{\textdef}[1]{\textit{#1}\index{#1}}

\DeclareMathOperator{\cK}{{\mathcal K}}

\newcommand{\KK}{{\mathcal K} }

\newcommand{\cE}{{\mathcal E} }

\newcommand{\SC}{\mathcal{S}^n_C}
\newcommand{\Sh}{{\mathcal S}^n_H}

\newcommand{\En}{{{\mathcal E}^n} }
\newcommand{\Ek}{{{\mathcal E}^k} }

\newcommand{\cL}{{\mathcal L} }

\newcommand{\cI}{{\mathcal I} }

\newcommand{\NEDM}{\textbf{NEDM}\,}
\newcommand{\NEDMp}{\textbf{NEDM}}
\newcommand{\EDM}{\textbf{EDM}\,}
\newcommand{\EDMp}{\textbf{EDM}}

\newcommand{\MIP}{\textbf{MIP}\,}
\newcommand{\MIPp}{\textbf{MIP}}

\newcommand{\SBGTp}{\textbf{SBGT}}

\newcommand{\BIEVp}{\textbf{BIEV}}
\newcommand{\MBFV}{\textbf{MBFV}\,}
\newcommand{\MBFVp}{\textbf{MBFV}}
\newcommand{\SDP}{\textbf{SDP}\,}
\newcommand{\SDPp}{\textbf{SDP}}

\newcommand{\DC}{\textbf{D\&C}\,}
\newcommand{\DCp}{\textbf{D\&C}}
\newcommand{\FV}{\textbf{FV}\,}
\newcommand{\FVp}{\textbf{FV}}
\newcommand{\FR}{\textbf{FR}\,}
\newcommand{\FRp}{\textbf{FR}}

\newcommand{\Rtn}{{\R^{\scriptsize{t(n)}}\,}}
\newcommand{\Rtno}{{\R^{\scriptsize{t(n-1)}}\,}}

\newcommand{\MM}{{\mathcal M}}

\newcommand{\cN}{{\mathcal{N}}}

\newcommand{\A}{{\mathcal A}}

\newcommand{\bbm}{\begin{bmatrix}}
\newcommand{\ebm}{\end{bmatrix}}
\newcommand{\bem}{\begin{pmatrix}}
\newcommand{\eem}{\end{pmatrix}}
\newcommand{\beqs}{\begin{equation*}}
\newcommand{\bet}{\begin{table}}
\newcommand{\eeq}{\end{equation}}
\newcommand{\eeqs}{\end{equation*}}
\newcommand{\beqr}{\begin{eqnarray}}

\DeclareMathOperator{\face}{face}

\DeclareMathOperator{\Null}{null}
\DeclareMathOperator{\nul}{null}
\DeclareMathOperator{\range}{range}
\DeclareMathOperator{\embdim}{edim}

\DeclareMathOperator{\gal}{{gal}}

\DeclareMathOperator{\blkdiag}{{blkdiag}}

\DeclareMathOperator{\diag}{{diag}}
\DeclareMathOperator{\Diag}{{Diag}}

\DeclareMathOperator{\offDiag}{{offDiag}}
\DeclareMathOperator{\usMat}{{usMat}}
\DeclareMathOperator{\usvec}{{usvec}}
\DeclareMathOperator{\svec}{{svec}}

\DeclareMathOperator{\sMat}{{sMat}}

\DeclareMathOperator{\relint}{{relint}}

\newcommand{\nc}{\newcommand}

\nc{\Mm}{{\mathcal M}^{m} }
\nc{\Mmn}{{\mathcal M}^{mn} }
\nc{\Mnr}{{\mathcal M}_{nr} }
\nc{\Mnmr}{{\mathcal M}_{(n-1)r} }
\nc{\kwqqp}{Q{$^2$}P\,}
\nc{\kwqqps}{Q{$^2$}Ps}

\nc{\notinaho}{(X,S)\in \overline{AHO}(\A)}
\nc{\inaho}{(X,S)\in AHO(\A)}

\newcommand{\bea}{\begin{eqnarray}}%
\newcommand{\eea}{\end{eqnarray}}%
\newcommand{\beas}{\begin{eqnarray*}}%
\newcommand{\eeas}{\end{eqnarray*}}%
\newcommand{\Rnn}{\R^{n \times n}}%
{}

\newcommand{\Hnp}[1][]{\,\mathbb{H}_+^{\ifthenelse{\equal{#1}{}}{n}{#1}}}
\newcommand{\Hn}[1][]{\,\mathbb{H}^{\ifthenelse{\equal{#1}{}}{n}{#1}}}
\newcommand{\Hk}[1][]{\,\mathbb{H}^{\ifthenelse{\equal{#1}{}}{k}{#1}}}
\newcommand{\Dn}[1][]{\,\mathbb{D}^{\ifthenelse{\equal{#1}{}}{n}{#1}}}

\newcommand{\cP}{\mathcal{P}}

\newcommand{\floor}[1]{\lfloor#1\rfloor}

\usepackage[makeroom]{cancel} 

\begin{document}

\title{
Single Element Error Correction
	in a Euclidean Distance Matrix
}

\newcommand*\samethanks[1][\value{footnote}]{\footnotemark[#1]}

\author{
	\href{https://www.uwindsor.ca/science/math/716/faculty-dr-y-alfakih}
	{Abdo Alfakih}\thanks{Dept. Math. \& Stat., Univ. of Windsor,
	Canada} \and
	\href{https://uwaterloo.ca/combinatorics-and-optimization/contacts/woosuk-jung}
	{Woosuk L. Jung}\thanks{
    Department of Combinatorics and Optimization, Faculty of Mathematics, University of Waterloo, Waterloo, Ontario, Canada N2L 3G1; Research supported by The Natural Sciences and Engineering Research Council of Canada.}
	\and
   \href{http://www.math.uwaterloo.ca/~hwolkowi/} {Henry
   Wolkowicz\samethanks}%
   \and
\href{https://uwaterloo.ca/combinatorics-and-optimization/contacts/tina-xu}
  {Tina Xu}\samethanks
}

\date{\today}
          \maketitle

\abstract{
We consider the \emph{exact} error correction of
a noisy Euclidean distance matrix, \EDMp, where the elements are
the squared distances between $n$ points in $\R^d$. 
For our problem we are given two facts:
(i) the embedding dimension, $d=\embdim(D)$, 
(ii) \emph{exactly one} distance in the data is corrupted by 
\emph{nonzero noise}.
But we do \underline{not} know the magnitude nor position of the noise. 
Thus there is a combinatorial element to the problem.
We present three solution techniques.
These use three divide and conquer strategies in combination
with three versions of facial reduction that use:  
exposing vectors, facial vectors, and Gale transforms. 
This sheds light on the connections between the various forms
of facial reduction related to Gale transforms.
Our highly successful empirics confirm the success of these approaches 
as we can solve huge problems of the order of 
$100,000$ nodes in approximately one minute to machine precision.
\\Our algorithm depends on identifying whether a principal submatrix of
the \EDM contains the corrupted element. We provide a theorem for doing
this that is related to the existing results for identifying \emph{yielding}
elements, i.e.,~we provide a characterization for guaranteeing the perturbed
\EDM  remains an \EDM with embedding dimension $d$. The characterization
is particularly simple in the $d=2$ case.
\\In addition, we characterize when the intuitive
approach of the nearest \EDM problem, \NEDMp,
solves our problem. In fact, we show that this happens if, and only if,
the original distance element is $0$, degenerate, and the perturbation is
negative.
}



\tableofcontents
\addtocontents{loa}{\def\string\figurename{Algorithm}}
\listoffigures
\listoftables

\section{Introduction}
We consider error correction for a given 
\textdef{Euclidean distance matrix, \EDMp}, $D$, where the 
elements are the squared distances between $n$ 
points in $\R^d$ with $d$ the embedding dimension. 
For our problem, we know the embedding dimension, and
it is given that \emph{exactly one} distance is corrupted with
nonzero noise.
But we do not know the magnitude or position of the noisy distance. 
Thus our problem is a \EDM completion problem with known embedding
dimension where the position of the unknown element is not given.
This means that there is a hard combinatorial element to the problem.
We use three different \textdef{divide and conquer, \DCp}, 
approaches in combination with
three different versions of \textdef{facial reduction, \FRp}, 
one of which involves Gale transforms. These methods
solve this problem efficiently and accurately in all but two hard cases.
\index{\DCp, divide and conquer} 
\index{\FRp, facial reduction} 
\index{$\embdim(D)$, embedding dimension of $D$} 
\index{embedding dimension of $D$, $\embdim(D)$} 
The results are related to \textdef{yielding} in an \EDM
studied in~\cite{MR3842577}, i.e.,~characterizing when a perturbation 
of one element yields an \EDMp. We also show that the intuitive approach of
finding the nearest \EDM only works in the trivial case where the
original distance is $0$, degenerate, and the perturbation is negative.

The first \DC approach divides the matrix into two overlapping
principal submatrix blocks. We then either identify a block that
contains the corrupted element and continue with the smaller problem
for that block; otherwise we use \EDM completion to identify the
correct position and magnitude of the noise. The second approach divides the
matrix into the largest number of overlapping principal submatrix blocks 
with minimal overlap that allows for \EDM completion. Again, we either
identify one block to further subdivide the problem
or we complete the \EDM to locate the correct position and
identify the noise. The third
approach divides the matrix into principal submatrix blocks of
smallest size that can be completed and we again apply the divide and
conquer approach.

The three \DC approaches are used in combination with
three different types of facial reduction: 
(i) finding exposing vectors $Z_i$ for each  block and then using the
exposing vector $Z=\sum_i Z_i$ to
identify the face  $f = \Snp\cap Z^\perp \unlhd \Snp$; 
(ii) finding a \emph{facial vector,
\FV}, denoted $V_i$, for each block, and then finding the intersection
$\range(V) = \cap_i \range(V_i)$ efficiently so as
to identify the face $f = V\Sc^d_+V^T \unlhd \Snp$, of the
semidefinite cone $\Snp$;  
(iii) a combination of exposing and facial vectors
first finds a Gale matrix $N$ so that $Z=NN^T$ 
is an exposing vector and yields a
facial vector $V$ to identify the face $f$. Identifying the correct face
allows one to find the correct full rank factorization of the Gram
matrix to complete the \EDMp.

\index{\EDMp, Euclidean distance matrix}
\index{$\Snp$, positive semidefinite}
\index{positive semidefinite, $\Snp$}
\index{$\Sn$, symmetric matrices}
\index{symmetric matrices, $\Sn$}
\index{\FRp, facial reduction}
\index{$\FVp$, facial vector}
\index{facial vector, $\FVp$}
\index{$Z$, exposing vector}
\index{exposing vector, $Z$}
\index{$N$, Gale matrix}
\index{Gale matrix, $N$}

Included are several interesting results on: the commutativity of 
some operators for this distance geometry problem; and properties of
having exactly one noisy element.
Our main result is three successful algorithms that efficiently find the
position and magnitude of the corrupted distance. The algorithms scale
well. The best algorithm can
solve problems of the order of $100,000$ points in approximately one
minute, and to machine precision. With embedding dimension
$d$, the data is then size $100,000d$ with resulting $D\in \EDM$ of size
$n=100,000$. Theoretical difficulties arise when the points are
not in \emph{general position}. In fact, we prove that difficulties
arise only when all but at most $d+1$ points
are in a linear manifold of dimension $d-2$.

In the process of this study, we
review the relationships between three methods of facial reduction and
Gale transforms. In addition, we show that the only time that the
intuitive approach, the nearest \EDM problem,
solves our problem is the case where the original distance is $0$, and
the perturbation is negative, i.e.,~a trivial case.

\subsection{Related Literature}
The literature for \EDM and more generally \emph{distance geometry}
is vast and includes many surveys and books with many applications in
multiple areas of mathematics, engineering, health sciences.
Classical results on \EDM completion problems appear in
e.g.,~\cite{MR96g:15031,MR1321802,joh90,lau98,MR90m:15010} and
relate to applications in protein folding and
molecular conformations, e.g.,~\cite{MR2274505,MR2810897},
as well as sensor network localization,
e.g.,~\cite{MR2274505,kriswolk:09}.

Completing a partial \EDM using facial reduction, \FRp, 
and exploiting cliques and facial vectors was presented in
\cite{kriswolk:09}. This approach was extended to the noisy case in
\cite{ChDrWo:14} by exploiting the exposing vector representation for
\FRp. Other approaches that do not necessarily use
\FR are discussed in
e.g.,~\cite{AlfakihAnjosKPW:08,DiKrQiWo:08,KrPiWo:06} and the more
recent surveys \cite{Dokmanic_2015,surv} and the references therein.
The paper \cite{MR4666877} discusses perturbation analysis for \EDM
completion and convergence of algorithms under perturbations.
Matrix completion with noise and \emph{outliers} is 
discussed in~\cite{wong2017matrix}.
Further applications of \EDM for finding the so-called 
\emph{kissing number} are found in 
sphere packing that has further applications
to error correcting codes from the fields of communications.
A recent discussion on \EDM completion with noisy data is
in~\cite{denoiseIEEE2024}.

An application to finding a single node in the graph where the edges to
the node are known but can be noisy was done
in~\cite{SreWangWolk:17}. This is closely related to our problem and the
paper includes applications to real world problems, e.g.,~when the 
unknown node is a cellular phone and the other nodes are cellular towers.
Our application would consider the case when one tower is partially
obstructed and provides a noisy distance to the cellular phone.
Further results related to our problem is~\cite[Section 7.1]{alfm18}
that deals with the completion problem
for \emph{one} missing entry of an \EDMp. See also~\cite{MR1321802}.

Further, we relate \FR to Gale transforms. 
Details for Gale transorms are presented
in~\cite{Gale:56} and discussed further in e.g.,~\cite{MR1869424}.

\subsection{Outline and Main Results}
In \Cref{sect:backgrProb} we present the basic properties of \EDM and
the main problem description. 
This includes the relationship with
semidefinite programming, \SDP and the Gram matrix. Further definitions
and results are presented as needed.

In addition, in~\Cref{sect:NEDM} we show that the only case when the
nearest \EDM problem solves our problem is when the original element
(distance) is zero and the perturbation is negative.

In \Cref{sect:backgrndEDMnFR} we present the background on \FR and
results on commutativity involving the Lindenstrauss operator

In \Cref{sect:threemethods} we present the three-by-three algorithms:
\Cref{sect:twoblks} presents the bisection \DC with
facial reduction using exposing vectors, \BIEVp;
\Cref{sect:multblcks} presents the multi-block case with \FVp s, \MBFVp, 
and is based on an
efficient technique for finding the intersection of faces.
A proof of finite convergence for
this algorithm can be given based on the characterization of
so-called \textdef{restricted yielding} in~\Cref{thm:genposzerosgenerald}.
Finite convergence for the other approaches can be proved similarly.
\Cref{sect:Gale} presents an equivalent approach that uses
small blocks with Gale
transforms, \SBGTp. This section also highlights the interesting relationships,
equivalences, between facial reduction and Gale transforms.

In \Cref{sect:empricseff} we present a complexity analysis of the 
methods as well as present the empirical results. We solve small, large,
huge problems.

In \Cref{sect:hardcases} we discuss the two hard cases where the
algorithms can fail. We present
a characterization in~\Cref{thm:genposzerosgenerald}
for detecting one noisy distance.
Here we extend the results on \textdef{yielding} 
in~\cite[Section 7.2]{alfm18}, i.e.,~characterizing when the \EDM
property is preserved under a single perturbation. We provide a
characterization when a perturbation yields (preserves) not only
an \EDM but also maintains embedding dimension $d$. The result states
that all but two points are on a linear manifold of dimension
$d-2$; and this is particularly elegant for $d=2$. We include empirics
for the hard case.

Our concluding remarks are in~\Cref{sect:concl}.

\section{Problem Description, \EDM and \FRp, Theoretical Results}
\label{sect:backgrProb}
We consider matrices \textdef{$S \in \Sn$} the space of 
$n\times n$ symmetric matrices equipped with the \textdef{trace inner
product} $\langle S,T\rangle = \trace ST$; 
we use \textdef{$\diag(S)\in \Rn$} to denote 
the diagonal of $S$; the adjoint mapping is
\textdef{$\diag^*(v)=\Diag(v)\in \Sn$}.
For positive integers $j,k$, we let 
\textdef{$[k] = 1,2,\ldots,k$}; \textdef{$[j,k] = j,j+1,\ldots,k$}.
Moreover, we denote \textdef{$E_{jk} = e_je_k^T+e_ke_j^T$, unit matrix}.
For index subset $\alpha \subset [n]$ we let $D_\alpha $ 
denote the corresponding principal submatrix of $D$, and
\index{unit matrix, $E_{jk} = e_je_k^T+e_ke_j^T$}
\[
\cP_\alpha: \Sn \to \Sc^\alpha, \, \cP_\alpha(D) = D_\alpha,
\]
is the projection, or coordinate shadow, corresponding to $\alpha$.
We abuse notation by using $\Sc^\alpha$ rather than $\Sc^{|\alpha|}$ to
emphasize the actual coordinates used. The inverse image is 
\[
\textdef{$\cP_\alpha^{-1}(\bar D) = \{D\in \Sn : D_\alpha = \bar D\}$}. 
\]
The cone of positive semidefinite matrices is denoted $\Snp \subset
\Sn$, and we use $X\succeq 0$ for $X\in \Snp$. Similarly, for positive
definite matrices we use $\Snpp, X\succ 0$. 
\index{$\cP_\alpha(D) = D_\alpha$, coordinate shadow}
\index{coordinate shadow, $\cP_\alpha(D) = D_\alpha$}

\index{configuration matrix, $P$}
\index{$P$, configuration matrix}
For a set of $n$ points $p_i\in \R^d, i=1,\ldots,n$,
we denote the matrix of points, the configuration matrix,
$P=\begin{bmatrix} p_1 &
p_2 & \ldots & p_n \end{bmatrix}^T\in \R^{n\times d}$. 
The Euclidean distance matrix (of squared distances) is $D =
\left(\|p_i-p_j\|^2\right) \in \En\subset \Sn$. We denote the
closed convex \textdef{cone of \EDMp, $\En$}. 
Here $0<d=\embdim(D)$ is the embedding dimension of the \EDMp.
Without loss of generality, we can  translate the points so that
they are centered,
i.e., with \textdef{vector of ones, $e$}, we have $P^Te = 0$.
Note that $v:= \frac{1}{n}P^Te$ is the barycenter of the points.
The translation is then given by
\index{$\En$, cone of \EDMp} 
\begin{equation}
\label{eq:translateP}
	P^T \leftarrow P^T-ve^T.
\end{equation}
Our approach works well under the assumption that the points are in 
\textdef{general position},
i.e.,~no $d+1$ of them lie in a proper hyperplane; or, equivalently,
that every subset of $d+1$ points are affinely independent,
e.g.,~\cite{MR3842577}. 
This guarantees that our algorithms succeed as
it is based on identifying whether a Gram matrix corresponding to chosen
principal submatrices has the correct rank
$d$. However, we have modified our algorithm to allow for the case when
the general position assumption fails.

We denote the corresponding
\textdef{Gram matrix, $G=PP^T$}. Then the classical result of
Schoenberg~\cite{MR1501980} relates the matrix of squared distances, the
\textdef{Euclidean distance matrix, \EDMp}, 
with a Gram matrix by applying the \textdef{Lindenstrauss operator,
$\KK(G)$},
\begin{equation}
\label{eq:Lindens}
D = \KK(G) = \diag(G)e^T+e\diag(G)^T-2G\in \En.
\end{equation}
Moreover, this mapping is one-one and onto between the 
\textdef{centered subspace, $\SC$}, and \textdef{hollow subspace,
$\Sh$},
\index{$\SC$, centered} 
\index{$\Sh$, hollow} 
\[
\SC = \{X\in \Sn \,:\, Xe = 0\}, \quad
\Sh = \{X\in \Sn \,:\, \diag(X) = 0\}.
\]
We ignore the dimension $n$ when the meaning is clear.
Note that the centered assumption $P^Te=0 \implies G\in \SC$. Let 
\begin{equation}
\label{eq:defJ}
\textdef{$J = I - \frac 1n ee^T$} 
\end{equation}
be the orthogonal projection onto the
orthogonal complement $e^\perp$. Note that the translation 
in~\cref{eq:translateP} is equivalent to
\[
P\leftarrow JP = P -\frac 1n ee^TP.
\]
We let \textdef{$\offDiag$}$:\Sn\to \Sn$ 
denote the orthogonal projection onto $\Sh$, the 
hollow matrices. Then
the \textdef{Moore-Penrose generalized inverse, $\cK^\dagger$}, is
\index{$\cK^\dagger$, Moore-Penrose generalized inverse}
\[
\cK^\dagger (D) = -\frac 12 J\offDiag(D)J.
\]
\begin{lemma}
\label{lem:genpos}
For $D \in \En$ with $d=\embdim(D)$, 
the general position assumption is equivalent to
\begin{equation}
\label{eq:equivgenpos}
\alpha \subset [n], |\alpha|=k \geq d+1 \implies
\rank(\cK^\dagger(D_\alpha)) = d.
\end{equation}
\end{lemma}
Moreover, we find the orthogonal projection $\cP_{\range(K^\dagger)}$ as
\index{orthogonal projection, $\cP_{\range(K^\dagger)}$}
\index{$\cP_{\range(K^\dagger)}$ orthogonal projection}
\begin{equation}
\label{eq:projcentered}
 \cK^\dagger \cK (S) = JSJ.
\end{equation}
Note that the orthogonal projection satisfies
$\cP_{\range(K^\dagger)}= \cP_{\SC}$.
Throughout we abuse notation and use $\cK,\cP,J$ without indicating the
dimension of the space that is involved.



\index{\EDMp, Euclidean distance matrix}
\index{$\Sn$, symmetric matrices}
\index{symmetric matrices, $\Sn$}
\index{$\KK(G)$, Lindenstrauss operator}

\subsection{Main Problem Description and Model}
\label{sect:mainprob}
We begin with a description of the problem.
\begin{problem}[Find magnitude/location of noise in \EDMp]
\label{prob:noise}
Let $\hat D\in \En$, the cone of \EDMp s, $d=\embdim(\hat D)$.
Let $1\leq i < j \leq n$ be indices and $\alpha \in \R$ and define the
data $D = \hat D + \alpha  E_{ij}$ with the unit matrix $E_{ij} = e_ie_j^T + e_je_i^T$. Then:
\begin{quote}
Find the position $i,\!j$ and the value of the noise $\alpha$ from the
given \EDM $D$ with one noisy element.
\end{quote}
\end{problem}

\begin{remark}[Naive Discrete Model]
The fact that we have a single noisy element and a known embedding
dimension leads to the following rank constrained mixed integer \MIP and
\SDPp:
\begin{equation}
\label{eq:MIPmodel}
(\MIP,\SDP) \quad \begin{array}{rcl}
p^* = & \min &  \|H\circ (\cK(G) - D)\|^2_F
\\    & \text{s.t.} & G\succeq 0, \, \rank(G) = d
\\  &&    \diag(H) = 0 
\\  &&    H \in \Sn \cap \{0,1\}^{n\times n}
\\  &&    e^THe = n^2-n-2,
\end{array}
\end{equation}
where $H\circ D$ denotes the Hadamard product and $H$ is the adjacency
matrix of the complete graph where one edge is missing as characterized
by the last constraint.
Since exactly one element is noisy, we know that the optimal value $p^*
= 0$ when the correct position $i,j$ is chosen for $H$.
This problem looks extremely difficult as we deal with binary variables
as well as semidefinite and rank constraints. Our approach in this paper
takes advantage of the structure to get extremely accurate solutions
very quickly, while avoiding the difficult nature of the hard discrete
problem~\cref{eq:MIPmodel}.
\end{remark}

\index{\MIPp, mixed integer program}
\index{mixed integer program, \MIPp}
\index{$H\circ D$, Hadamard product}
\index{Hadamard product, $H\circ D$}

\subsection{Nearest \EDMp, \NEDMp, 
Formulation and Solving Main {\Cref{prob:noise}}}
\label{sect:NEDM}
It is natural to  try and solve~\Cref{prob:noise} by finding the nearest
\EDMp, (\NEDM, see~\cref{eq:NEDM} below)
to the given noisy $D_n$, i.e.,~relax the MIP 
formulation~\cref{eq:MIPmodel}. This emphasizes the difficulty 
of~\Cref{prob:noise} when
the noisy data $D_n\in \En$. We now see that the \NEDM approach works only in
the case where the corrupted element arises from a zero distance in the
original $D_0\in \En$, i.e.,~the degenerate case of one point on top of
another.

\index{\NEDMp, nearest \EDMp}
\index{nearest \EDMp, \NEDMp}

Let $V$ be chosen so that $[V~e]$ is an orthogonal matrix and 
for $X\in \Sc^{n-1}$ let \textdef{$\cK_V(X) := \cK(VXV^T)$}, 
where $\cK$ is given in \cref{eq:Lindens}. Then
\[
\textdef{$\cK^*(D) = 2(\Diag(De)-D)$}, \,\,
\textdef{$\cK_V^*(D) = V^T\cK^*(D)V$},
\]
Moreover, 
\[
0 = VXV^T \iff 0 = V^TVXV^TV \iff X=0,
\]
and therefore, for $D = \cK(VXV^T)$ we get that $VXV^T$ is a centered
matrix by definition of $V$, and so $D$ is a hollow matrix and
\begin{equation}
\label{eq:kvkvsposdef}
X\neq 0 \implies \langle X, \cK_V^*( \cK_V(X) ) \rangle
 = \| \cK_V(X)  \|^2_F = \| D \|^2_F \neq 0 \implies
\cK_V^* \cK_V \succ 0,
\end{equation}
i.e.,~\textdef{$\cK_V^*( \cK_V(\cdot)) \succ 0$}, is a positive definite linear
operator. The facially reduced \NEDM then has a unique solution and is:
\begin{equation}
\label{eq:NEDM}
\textdef{$\bar X(D_n)$} =  \bar X = 
          \argmin_{X\succeq 0} \frac 12 \|\cK_V(X)-D_n\|_F^2.
\end{equation}

\index{normal cone, $\cN_{\Snp}(\bar X)$}

\begin{lemma}[{\textdef{$\cN_{\Snp}(\bar X)$, normal cone}}]
Let 
\[
\bar X = 
\begin{bmatrix}
Q_1 & Q_2
\end{bmatrix}
\begin{bmatrix}
\bar \Lambda & 0 \cr 0 & 0
\end{bmatrix}
\begin{bmatrix}
Q_1 & Q_2
\end{bmatrix}^T
 \in \Snp, \,  \bar \Lambda \in \Sc^r_{++},
\]
with $Q = \begin{bmatrix} Q_1 & Q_2 \end{bmatrix}$ orthogonal. Let $\lambda
\in \Rnp$ be the vector of eigenvalues of $\bar X$. Then the normal cone
\[
\begin{array}{rcl}
\cN_{\Snp} (\bar X) 
&=& 
Q\Diag\left(\cN_{\Rnp}(\lambda)\right)Q^T
\\&=&
\{X \in -\Snp : X\bar X = 0 \}
\\&=&
\left\{X \in -\Snp : X = 
\begin{bmatrix}
Q_1 & Q_2
\end{bmatrix}
\begin{bmatrix}
0 & 0 \cr 0 & \hat \Lambda 
\end{bmatrix}
\begin{bmatrix}
Q_1 & Q_2
\end{bmatrix}^T, \hat \Lambda \in -\Sc^{n-r}_+ \right\}
\\&=&
-\left(\Snp - \bar X\right)^+
\\&=&
-\face(\bar X)^c,  \quad (\text{the conjugate face}).
\end{array}
\]
\end{lemma}
\begin{proof}
The proof follows from using the Eckart-Young-Mirsky,
e.g.,~\cite{EckartYoung:36}, and the Moreau decomposition, e.g.,~\cite{Mor:62}, 
theorems. That is $\bar X$
is the nearest point in $\Snp$ to a given $Y\in \Sn$ if, and only if, $Y-\bar X
\in \cN_{\Snp}(\bar X)$, and $Y$ has an orthogonal decomposition 
\[
Y = Y_++Y_-, \,\, Y_+,-Y_-\in \Snp, \,\, Y_+Y_- = 0,
\]
using the projections $Y_+,Y_-$ onto the nonnegative and nonpositive
semidefinite cones, respectively. Therefore $Y-Y_+ =  Y_-, Y_+=\bar X$
yields the result. Note that  the Eckart-Young-Mirsky theorem
characterizes the projections as being the same as using the Moreau
decomposition.
\end{proof}

We now present a characterization for when the nearest point finds the
original \EDM after a single element is corrupted. Note that for our
problem the $i,j$ are not known.
We first introduce the notion of a face and some its properties. 
For a given convex cone $K$, a convex cone $f\subset K$ 
is a \textdef{face of $K$}, denoted $f\unlhd K$, if 
\[
X,Y\in K, X+Y\in f \implies X,Y\in f.
\]
Given $f\unlhd K$, the \textdef{conjugate face, $f^c = K^+ \cap f^\perp$}.
For $X\in K$, we denote \textdef{$f=\face(X)$, minimal face of $K$},
i.e.,~the intersection of all faces that contain $X$.
For $X\in \Snp$ with spectral decomposition 
$X = \begin{bmatrix} V & U \end{bmatrix}
\begin{bmatrix} \Lambda & 0 \cr 0 & 0 \end{bmatrix}
\begin{bmatrix} V & U \end{bmatrix}^T,\,\Lambda \in \Sc^r_{++}$, we get
\[
\face(X) = V\Sc_+^r V^T, \, \face(X)^c = U\Sc_+^{n-r} U^T.
\]
See e.g.,~\cite{DrusWolk:16}.

\index{minimal face of $K$, $f=\face(X)$}
\index{$f^c = K^c \cap f^\perp$, conjugate face}
\begin{theorem}
\label{thm:charactNEDMequiv}
Suppose that $D_0\in \En$, the perturbation matrix is $E_{ij}\in \Sn, i<j$, 
and $D_n = D_0+\alpha E_{ij}\notin \En$. 
Let
\[
\begin{array}{rcl}
Y:=\cK_V^*(E_{ij}) 
&=&
 V^T K^*(E_{ij}) V
\\&=&
2 \begin{bmatrix}
v_i^Tv_i+ v_j^Tv_j -v_i^Tv_j -v_j^Tv_i
\end{bmatrix},
\end{array}
\]
where $v_i$ is the $i$-th row of $V$.  Then $\alpha \neq 0$, 
$0\neq Y\succeq 0$, and
\begin{equation}
\label{eq:X0conditionopt}
\begin{array}{rcl}
\text{$D_0$ is the nearest \EDM to $D_n$}
&\iff&
D_0 = \cK_V(\bar X), \text{ for $\bar X$ from~\cref{eq:NEDM}},
\\&\iff&
X_0:=\cK_V^\dagger(D_0) \in Y^c := \Sc^{n-1}_+\cap Y^\perp, 
\text{  and  } \alpha < 0,
\\&\iff&
X_0:=\cK_V^\dagger(D_0) \in Y^c := \Sc^{n-1}_+\cap Y^\perp, 
\,\, \forall \alpha < 0,
\\&\iff&
(D_0)_{ij} = 0, \,\, \text{  and  } \alpha < 0.
\end{array}
\end{equation}

\end{theorem}
\begin{proof}
We consider the nearest \EDM problem in~\cref{eq:NEDM}
with $X_0 =  \cK_V^\dagger(D_0)$ as the original data.
We want to characterize $D_0$, or equivalently $X_0$, so that the
optimum $\bar X$  in~\cref{eq:NEDM} equals $X_0$, i.e,~the \NEDM problem
solves our~\Cref{prob:noise}.

First we note from~\cref{eq:kvkvsposdef} that 
\textdef{$\cK_V^* \cK_V \succ 0$} so that the objective 
in~\cref{eq:NEDM} is strongly convex and we have a unique optimum.
The optimality conditions for $X_0$ for the nearest point
problem are that the gradient of the objective function is in the polar
cone (negative normal cone) of the \SDP cone at $X$, i.e.,~we have the
equivalences
\[
\begin{array}{rcl}
\cK_V^*(\cK_V(X_0)-D_n)  \in (\Sc^{n-1}_+ - X_0)^+
&\iff &
\cK_V^*(\cK_V\cK_V^\dagger (D_0)-D_n)  \in (\Sc^{n-1}_+ - X_0)^+
\\&\iff &
\cK_V^*(D_0-D_n)  \in (\Sc^{n-1}_+ - X_0)^+
\\&\iff &
\alpha\cK_V^*(-E_{ij})  \in (\Sc^{n-1}_+ - \cK^\dagger_V(D_0))^+
\\&\iff &
-\alpha\cK_V^*(E_{ij})  \in (\Sc^{n-1}_+ - \cK^\dagger_V(D_0))^+.
\end{array}
\]

Note that $\range(\cK) = \range(\cK_V) = \Sh$. Therefore $\Null(\cK^*_V)
= (\Sh)^\perp = \range(\Diag)$. As $E_{ij}$ is not diagonal, we get
$\cK_V^*(E_{ij})\neq 0$. Alternatively,
we recall that the choice of $V$ implies $\bar V := \begin{bmatrix}
V & \frac 1{\sqrt n} e\end{bmatrix}$ is orthogonal. Therefore, the
rows satisfy 
\[
v_i^Tv_j + \frac 1n = 0, v_i^Tv_i + \frac 1n = 1 \implies
v_i^Tv_i+ v_j^Tv_j -v_i^Tv_j -v_j^Tv_i = 2\left(1-\frac 1n +\frac
1n\right) = 2\neq 0,
\]
i.e.,~we see again that $\cK_V^*(E_{ij})\neq 0$. 

We now observe 
that necessarily $\alpha \leq 0$. Without loss of generality, we assume
$\alpha < 0$.  To have the nearest matrix be the original data $D_0,X_0$, 
we need to have
\[
\begin{array}{rcl}
\cK_V^*(E_{ij})  
&\in & 
(\Sc^{n-1}_+ - \cK^\dagger_V(D_0))^+ 
\\&=&
(\Sc^{n-1}_+ - X_0)^+
\\&=&
 \{X \in \Sc^{n-1}_+ : XX_0 = 0 \}
\\&=&
\face(X_0)^c,
\end{array}
\]
i.e.,~the conjugate face to $\face(X_0)$.

\index{conjugate face, $\face(X_0)^c$}
\index{$\face(X_0)^c, $conjugate face}

Next we see the equivalence using $v_i$, the rows of $V$:
\[
\begin{array}{rcll}
Y=\cK_V^*(E_{ij}) 
&=&
 V^T K^*(E_{ij}) V
\\&=&
2V^T (\Diag(E_{ij}e) - E_{ij})V, &  (\succeq 0,
      \text{  since congruence of a Laplacian matrix})
\\&=&
2 V^T (e_ie_i^T + e_je_j^T - e_ie_j^T + e_je_i^T)V  
\\&=&
2 \begin{bmatrix}
v_i^Tv_i+ v_j^Tv_j
-v_i^Tv_j -v_j^Tv_i
\end{bmatrix}.
\end{array}
\]
Since $Y, X_0\in \Sc^{n-1}_+$, we have 
\[
Y\in \face(X_0)^c \iff 
X_0 \in \face(Y)^c \iff YX_0 = 0  \iff \trace YX_0 = 0.
\]
Therefore, the first two equivalences in~\cref{eq:X0conditionopt}
follow as $D_0 = K_V(X_0)$.

Finally, we note that the conjugate face condition is equivalent to
\[
 0= \trace X_0Y = \trace X_0 V^T K^*(E_{ij}) V
 =\trace \cK(VX_0 V^T) E_{ij} = \trace D_0E_{ij}.
\]
Therefore, we can only have a negative perturbation from $(D_0)_{ij}=0$.
\end{proof}

\subsection{Results on \EDMp, \FRp}
\label{sect:backgrndEDMnFR}
We now present know results for \EDM and \FR as well as new results
related to our specific problem.
In particular,~\Cref{lem:CommutePK} provides a useful commutativity
result for $\cK,\cP_\alpha$, while~\Cref{lem:sumexposedvecs} recalls the
result that a sum of exposing vectors is an exposing vector.
\Cref{thm:FRone} shows how to find a (centered) \FV for the face
corresponding to a given principal submatrix $D_\alpha$. 
Then~\Cref{lem:twocliquered} illustrates an efficient and accurate way
to find the intersection of two given faces by finding a \FV from the
corresponding two given \FVp s. 

The results here use the \emph{general position} assumption.
Further results are in~\Cref{thm:genposzerosgenerald} that
characterizes when the perturbation of a single element in
an \EDM with embedding dimension $d$
can still be a \EDM with embedding dimension $d$. This is important for
our general algorithm in identifying the location of the corrupted element in
the \EDMp.

\subsubsection{Commutativity of $\cP_\alpha,\cK$}
We first include a useful and very interesting
observation about the commutativity of $\cK,
\cP_\alpha, \alpha \subset [n]$.
\begin{lemma}
\label{lem:CommutePK}
Let $M\in\Rnn, \alpha \subset [n]$. Then, by abuse of notation on the
dimensions of $e$ and the transformations $\cK,\cP_\alpha$, we get
\[
\cP_\alpha \cK (M) = \cK \cP_\alpha(M).
\]
\end{lemma}
\begin{proof}
We note that $\cP_\alpha(\diag(M)) = \diag(\cP_\alpha(M))$. Therefore,
\[
\begin{array}{rcl}
\cP_\alpha \cK (M) 
&=&
  \cP_\alpha \left( \diag(M)e^T + e\diag(M) -2M\right)
\\&=&
   \diag(\cP_\alpha (M))e^T + e(\diag(\cP_\alpha M)) -2\cP_\alpha\left(M\right)
\\&=&
 \cK \cP_\alpha(M).
\end{array}
\]
\end{proof}

We note that symmetry for $M$ is not needed. Moreover, the commutativity
reveals information on the eigenspace of $\cK$, for if it was an
operator we would  have joint diagonalization with $\cP$.

\subsubsection{Facial Reduction}
\index{face of $\Snp$, $f\unlhd \Snp$}
\index{$f\unlhd \Snp$, face of $\Snp$}
We follow notation and definitions in~\cite{DrusWolk:16,alfm18}. First
we recall that both $\Snp, \En$ are closed convex cones in $\Sn$.
We let $W,V$ be two full column rank matrices that satisfy
\[
\range(V) = \range(X), \, \range(W) = \nul(X),\,  Z=WW^T, \, d = \rank(X).
\]
Then $X\in \relint (f)$ and we have two representations for $f$:
\[
f = V\Sc^d_+V^T = \Snp \cap Z^\perp.
\]
We call $V,Z$ a facial and exposing vector, respectively.
\index{$V$, facial vector} 
\index{facial vector, $V$} 

If we choose the facial vector to satisfy 
$\begin{bmatrix} V & e \end{bmatrix}$ nonsingular and $V^Te = 0$, then we
can characterize the face of centered Gram matrices
\[
\SC\cap \Snp = V\Sc^{n-1}_+V^T = \cK^\dagger (\En) \unlhd \Snp.
\]
This is used in~\cite{AlKaWo:97} to regularize the \EDM completion
problem using $\cK_V(X) := \cK(VXV^T)$, i.e.,~$\cK_V : \Sc_+^{n-1} \to
\En$ and strict feasibility is satisfied for \EDM completion problems.
Here $V$ is a \textdef{centered facial vector}, i.e.,~$V^Te=0$. Since we
work with centered Gram matrices, we often use centered facial
vectors below.

Our divide and conquer methods use principal submatrices $D_\alpha$ of
$D$ and corresponding Gram matrices, with
ordered integers $\alpha = [i,i+k], 1\leq i\leq i+k\leq n$. The first \FR
method finds exposing vectors for each submatrix and then adds
these up to get an exposing vector for the entire Gram matrix $G$,
i.e.,~we exploit~\Cref{lem:sumexposedvecs}. Note that if the overlap of
the exposing vectors is deficient, then it is not necessarily true that
$Z$ is a maximum rank exposing vector.
\begin{lemma}[\cite{DrusWolk:16}]
\label{lem:sumexposedvecs}
Let 
\[
G\in \Snp, \, Z_i  \in \Snp, \trace (GZ_i) = 0, i = 1,\ldots,k.
\]
Then $Z=\sum_{i=1}^k Z_i \succeq 0$ and  $GZ = 0$.
\end{lemma}

The second \FR method uses adjacent pairs of principal submatrices in
order to do \FR by intersecting pairs of faces.
We first consider the representation of \FR for a single principal
submatrix, i.e.,~the representation of the smallest face obtained using
the Gram matrix corresponding to $D_\alpha$.
Without loss of generality we use the first, top left,
principal submatrix. We modify the notation in~\Cref{thm:FRone}
to match the notation herein. Recall that $\cP^{-1}_\alpha$ denotes the 
projection inverse image.
\index{$\cP^{-1}_\alpha$, projection inverse image}
\index{projection inverse image, $\cP^{-1}_\alpha$}
\begin{theorem}[{\cite[Theorem 2.3]{kriswolk:09}}]
\label{thm:FRone}
Let 
\[
D\in \En, d = \embdim(D),\, \alpha =[k],\, \,
\bar D = D_\alpha, t=\embdim(\bar D).
\]
Let 
\[
\bar G= \cK^\dagger (\bar D) = \bar U_G S \bar U_G^T, \, S\in
\Sc^t_{++},
\]
where
\[
\bar U_G\in \R^{k\times t}, \, \bar U_G^T\bar U_G = I_t, \,\bar
U_G^Te=0.
\]
Furthermore, let 
$U_G := \begin{bmatrix} \bar U_G & \frac 1{\sqrt k} e \end{bmatrix}$,
$U := \begin{bmatrix} U_G & 0\cr 0 & I_{n-k} \end{bmatrix}$,
and 
$\begin{bmatrix} V & \frac 1{\|U^Te\|} U^Te \end{bmatrix}
\in \R^{n-k+t+1 \times n-k+t+1}$ be orthogonal. Then
\[
\face \left( \cK^\dagger \left(\cP_\alpha^{-1} (\bar D)\cap \En
  \right)\right) = \left(U\Sc_+^{n-k+t+1}U^T\right)\cap \SC =
\left((UV)\Sc_+^{n-k+t}(UV)^T\right).
\] 

\end{theorem}

The matrix $UV$ in~\Cref{thm:FRone} provides a facial vector for the 
minimal face corresponding to the block $\bar D$.
If we have two overlapping principal submatrices with the overlap having
the proper embedding dimension $d$, then we can efficiently find a
facial vector for the intersection of the two corresponding faces. This
is given in \cite[Lemma 2.9]{kriswolk:09} and we have added that option
to our code and include the details here in~\Cref{lem:twocliquered} for
completeness. 

\begin{lemma}[{\cite[Lemma 2.9]{kriswolk:09}}]
\label{lem:twocliquered}
Let
\[
U_1:= \kbordermatrix{ 
       & r+1 \cr
  s_1  & U^{\prime}_1  \cr 
  k    & U^{\prime\prime}_1 
}, \quad
U_2:= \kbordermatrix{
      & r+1 \cr
  k   & U^{\prime\prime}_2 \cr 
  s_2 & U^{\prime}_2
}, \quad
\hat U_1:= \kbordermatrix{
      & r+1  & s_2 \cr
  s_1 & U_1'  & 0   \cr 
  k   & U_1'' & 0   \cr 
  s_2 & 0     & I
}, \quad
\hat U_2:= \kbordermatrix{ 
      & s_1 & r+1  \cr
  s_1 & I   & 0 \cr 
  k   & 0   & U_2'' \cr
  s_2 & 0   & U_2'
}
\]
be appropriately blocked with $U_1'', U_2'' \in \MM^{k \times (r+1)}$
full column rank and $\range(U_1'') = \range(U_2'')$.
Furthermore, let
\begin{equation}
\label{eq:U1U2}
\bar U_1 := 
\kbordermatrix{ 
      & r+1 \cr
  s_1 & U_1^\prime\cr
  k   & U_1^{\prime\prime}\cr
  s_2 & U_2^{\prime} (U_2^{\prime\prime})^\dagger U_1^{\prime\prime}
}, \quad
\bar U_2 := 
\kbordermatrix{ 
      & r+1 \cr
  s_1 & U_1^{\prime} (U_1^{\prime\prime})^\dagger U_2^{\prime\prime}\cr
  k   & U_2^{\prime\prime}\cr
  s_2 & U_2^\prime
}.
\end{equation}
Then $\bar U_1$ and $\bar U_2$ are full column rank and satisfy
\[
\range (\hat U_1) \cap  \range (\hat U_2) =   
\range\left(\bar U_1 \right) = \range\left(\bar U_2 \right).
\]
Moreover, if $e_{r+1} \in \R^{r+1}$ is the $(r+1)^\mathrm{st}$ standard unit vector, and
$U_ie_{r+1} = \alpha_i e$, for some $\alpha_i \neq 0$, for $i=1,2$, then
$\bar U_ie_{r+1} = \alpha_i e$, for $i=1,2$.
\end{lemma}

\begin{remark}
\label{rem:bestoverlap}
In~\Cref{lem:twocliquered}, if we have two overlapping blocks 
$D_{s_1\cup k}, D_{k \cup s_2}$ we can get facial vectors from bases of
the corresponding Gram matrices in  $U_1,U_2$, respectively. These are
matched up in $\hat U_1,\hat U_2$. Then a facial vector for the
intersection of the minimal faces containing these blocks is given in
either of the formulae in \cref{eq:U1U2}. 
In our implementation, we
choose the one for which $U^{\prime\prime}_1, U^{\prime\prime}_2$ is better
conditioned. 
Thus we use the two facial vectors and find a new
facial vector for the intersection of the overlapping ranges.

This emphasizes the importance of the conditioning of the overlap in
$U^{\prime\prime}_1, U^{\prime\prime}_2$. In fact, it is essential that
the rank of each overlap is eventually $d$, the embedding dimension of
the problem. But there is no reason that
the ordering of the nodes (rows) of $D$ cannot be changed. Therefore, if
we find a well-conditioned block of correct rank $d$,  
we can include that in the overlap
for further iterations, i.e.,~we can always save and use the best conditioned
block with the largest rank in the overlap.
\end{remark}

\section{Three \DC and  Three \FR Methods}
\label{sect:threemethods}
We have three different \DC methods: bisection; minimum block
overlap; and minimum block size. And we have three different \FR
methods: exposing vector; facial vector; and Gale transform.
The result is nine possible algorithms. We now pair
each \DC method with exactly one \FR method and describe
and implement these three methods.
For simplicity, we assume that the \emph{general position} assumption
holds, i.e.,~we can easily find a principal submatrix of the Gram matrix
with rank $d$, see~\Cref{rem:overlapgenpos}.

\subsection{Bisection \DC with Exposing Vector for \FRp}
\label{sect:twoblks}
\index{\BIEVp, bisection with exposing vector}
\index{bisection with exposing vector, \BIEVp}
The first algorithm combines the bisection approach for
\DC with the exposing vector approach for \FRp, denoted \BIEVp.
We divide the data matrix $D$ into $2$ properly
overlapping blocks corresponding to principal
submatrices. We first identify whether the noisy element is within 
one of these two
principal blocks and reduce the problem to that block. 
Or if it is outside both chosen principal
blocks, then we apply \FR explicitly to solve the problem and find the
position $ij$ and noise $\delta$.


Recall that $d$ is the embedding dimension. We assume that $n \gg  d$,
i.e.,~is sufficiently larger than $d$.
Then we use the two blocks indexed by columns (rows)
\begin{equation}
\label{eq:twosets}
I_1 = \left\{1,\ldots,\ceil{(n+d+2)/2}\right\},
\quad
I_2 = \left\{\floor{(n-d-2)/2},\ldots,n\right\},
\end{equation}
i.e.,~we have the two principal submatrices 
$D_1=D_{I_1},D_2=D_{I_2}$ which overlap in the block of size at least
$\frac 12((n+d+2)-(n-d-2)) = d+2$. In addition, the assumption $n\gg d$
implies each block is at least size $d+1$.
\begin{remark}
\label{rem:overlapgenpos}
The overlap corresponds to the points 
\[I_1\cap I_2 = \left\{ \floor{(n-d-2)/2}, \ldots, 
	\ceil{(n+d+2)/2}\right\}.
\]
If the
corresponding (centered) Gram matrix has rank $d$, then we know that the
overlap of the graph is rigid. If this is not the case, then we need to
permute the columns of $D$ in order to obtain a rigid overlap.

This raises the \emph{hard} question of how to find the best overlap.
The simple case occurs if the overlap is in \textdef{general
position}, i.e.,~each principal submatrix $D_\alpha, \alpha 
\subset [n], |I| \geq d+1$ corresponds to a Gram matrix with rank $d$.
\end{remark}

We now find the \emph{supposed} Gram matrices $G_i = \cK^\dagger(D_i), i=1,2$,
i.e.,~these are indeed \emph{centered} 
Gram matrices with rank $d$  if $D_i$ is an \EDMp. 
There are now \underline{three cases} to consider for this current
approach.

\subsubsection{Case 1: Reducing Size of Problem}
\label{sect:casereduceblks}
Suppose that one of $G_i, i=1,2$, is not positive semidefinite of rank
$d$, i.e.,~as a result of the noise in the data,
the corresponding submatrix $D_i$ is not an \EDM or does not
have embedding dimension $d$. Then we can continue on the
\DC approach and reduce our problem to that
submatrix. And then we continue the division, i.e.,~we have reduced the
problem by a factor of roughly $2$. We then redefine $D$ and $n$
appropriately and return to dividing the indices in~\cref{eq:twosets}.

\subsubsection{Case 2: \EDM Completion using \FRp}
\label{sect:casecomplFR}
We first recall that
for $G\succeq 0$, \textdef{$\face(G)$} denotes the smallest face
containing $G$. Let $V,N$ denote matrices with columns that are
orthonormal basis for $\range(G),\nul(G)$, respectively, e.g.,~made up
of an orthonormal set of eigenvectors. Let $d=\rank(G)$. Then as stated
above,
\begin{equation}
\label{eq:facialexpos}
\face(G) = V\Sc^d_+ V^T = \Snp\cap (NN^T)^\perp.
\end{equation}
$V$ is a \textdef{facial vector}, \cite{ImWolk:22}, 
while $Z=NN^T$ is an \textdef{exposing vector}, e.g.,~\cite{ChDrWo:14}.
And moreover, the sum of exposing vectors is an exposing vector,
~\Cref{lem:sumexposedvecs}:
\begin{equation}
\label{eq:sumexposvctrs}
Z_i\succeq 0, \trace Z_iG=0, i=1,\ldots,k \,\,\implies \,\,
                     \sum_i Z_i \succeq 0,\,
                     \trace \left(\left(\sum_i Z_i\right)G\right)=0.
\end{equation}
Suppose both $G_i, i=1,2$, are \emph{centered}
Gram matrices with rank $d$. Then we know the
corrupted element/distance is outside the principal blocks. We now
continue to completely solve the problem using \FR as
we can now complete the partial \EDM formed from the two blocks. 
We now give the details.\footnote{See e.g.,~\cite[Algor. 1, Pg.
2308]{ChDrWo:14} for more details.}

\begin{enumerate}
\item
Using the spectral decomposition of $G_i,i=1,2$, given above,
we obtain orthonormal bases of \emph{centered} null vectors for
the nullspaces of the Gram matrices $G_i, i=1,2$,
\begin{equation}
\label{eq:GNInullsp}
G_iN_i = 0,\, N_i^Te=0, \, N_i^TN_i = I,\, i=1,2.
\end{equation}
Let $Z_i=N_iN_i^T\succeq 0, i=1,2$, be the corresponding 
\emph{centered} exposing vectors, i.e.,~we have
\index{exposing vector}
\index{exposing vector!centered}
\index{exposing vector!maximum rank}
\begin{equation}
\label{eq:Ziexposvect}
Z_iG_i = 0, G_ie=0,\, Z_ie=0, G_i+Z_i+ee^T \succ 0, \, i=1,2.
\end{equation}
\item
Let $W_i$ be zero matrices of order $n$ and set
\[
(W_i)_{I_i} = Z_i, \, i=1,2,
\]
i.e.,~we place the \emph{centered}
exposing vectors $Z_i$ into the correct blocks. Now each 
$W_i,i=1,2$, is a \emph{centered} exposing vector for the true 
centered Gram matrix $G$.
As the sum of exposing vectors is an exposing vector,
we form the \emph{centered} exposing vector of the true 
centered Gram matrix $G$,
\begin{equation}
\label{eq:sumWs}
Z = W_1+W_2.
\end{equation}
This yields a \textdef{maximum rank exposing vector}
\index{exposing vector!maximum rank}
\begin{equation}
\label{eq:GZeposdef}
GZ=0,\, Ge=0,\, Ze=0,\, Z\succeq 0,\, G+Z+ee^T\succ 0.
\end{equation}
\item
We choose $V, V^TV=I$, to be full column rank and to span $\Null
([Z~e]^T)$.\footnote{As noted following \cref{eq:facialexpos},
this is called a \emph{facial vector}
\cite{ImWolk:22}. In fact, this is a centered facial vector.} 
This completes the \FR as we have 
\begin{equation}
\label{eq:GVRVGE}
G = VRV^T,\, R\in \Sc_{++}^d, \, Ge = 0.
\end{equation}

\item
We denote the \textdef{adjacency matrix, $H_\alpha$} to be the matrix
of zeros with ones in the 
positions indexed by $I$. Recall that
$H\circ D$ denotes the Hadamard product.
We solve the cone least squares problem that does not include the noisy
element:
\index{$H_\alpha$, adjacency matrix}
\index{$H\circ D$, Hadamard product}
\index{Hadamard product, $H\circ D$}
\begin{equation}
\label{eq:lssR}
\min_{R\succeq 0}
\|H_{I_1\cup I_2} \circ \cK(VRV^T) - H_{I_1\cup I_2} \circ D\|_F^2.
\end{equation}
Since $R$ is order $d$ and at the start $d << n$, this can be a very 
overdetermined problem and
can be helped by using a \emph{sketch matrix}, see e.g.~\cite{ChDrWo:14}.

\item
In the case of random data, we expect all non-noisy blocks of size at
least $d+1$ to be proper \EDMp s. Therefore, the optimal solution $R$ in
\cref{eq:lssR} is unique and in $\Sc_{++}^d$. Therefore, we
can solve this least squares problem as an unconstrained problem and
improve on the accuracy and speed. There is one constraint, linear
equation, for each distinct pair in the set
\[
\textdef{$\cI_R$} := \left\{ (i,j) \,:\, 
(i,j) \in (I_1 \times I_1)\cup (I_2 \times I_2), \, i<j\right\}.
\]
We let $\svec : \Sn \to \Rtn$ denote the isometric mapping that vectorizes
a symmetric matrix  columnwise with multiplying the off-diagonal
elements by $\sqrt 2$. The inverse (and adjoint) is $\sMat = \svec^*$.
Similarly, $\usvec :\Sn \to \Rtno$ is for the strict triangular part
with $\usMat = \usvec^*$.
Let the unknown variable be $r = \svec(R) \in \R^{t(d)}$, 
where we denote \textdef{$t(d)$, triangular number}. The $m_R\times n_R$
system,
$m_R = t(|I_1|-1)+t(|I_2|-1)-t(|I_1\cap I_2|-1) \times (n_R := t(d))$ linear system
to solve for $r$ is
\index{triangular number, $t(d)$}
\index{$\svec$}
\index{$\sMat$}
\begin{equation}
\label{eq:eqnsfindR}
\usvec H_{I_1\cup I_2} \circ \cK(V\sMat(r)V^T) = \usvec H_{I_1\cup I_2} \circ D,
\quad ( : \R^{t(d)} \to \R^{t(n-1)}).
\end{equation}
The columns of the matrix representation are obtained by replacing $r$
with unit vectors $e_i, i = 1,\ldots t(d)$. We note that the rows
corresponding to $i,j$ not in the union of the two blocks are zero and
can be discarded.

Alternatively, we could find the transpose of the matrix representation
by taking the adjoint of the left-hand side
in \cref{eq:eqnsfindR} and working on $\R^{t(n)} \to \R^{t(d)}$.
We use $g = e_\iota\in \Rtno, \usMat(g)$. The adjoint is given by
\begin{equation}\label{eq:eqnsfindRT}
	\svec\left(V^T\cK^*\left[ H_{I_1\cup I_2}\circ \usMat(g) \right]V\right).
\end{equation}
We first simplify $H_{I_1\cup I_2}\circ \usMat(e_\iota)$ to get the row of the matrix representation. For each $\iota\in[t(n-1)]$, we can find $i,j\in [n]$ such that $\iota = i + \sum_{\ell = 0}^{j-2}\ell$ and observe
\[
H_{I_1\cup I_2}\circ \usMat(e_\iota) = \begin{cases}
	\frac{1}{\sqrt{2}}\left(e_ie_j^T + e_je_i^T\right) & \text{if }H_{ij}=1;\\
	0 & \text{otherwise},
\end{cases}
\]
where $H_{ij}$ denotes the $i,j$-th entry of $H_{I_1\cup I_2}$. 
Restricting $g$ to the case when $H_{ij}=1$, gives
\[
\begin{array}{rcl}
\cK^*\left[ H_{I_1\cup I_2}\circ \usMat(e_\iota) \right]
&=&
\cK^*\left[ \frac{1}{\sqrt{2}}\left(e_ie_j^T + e_je_i^T\right) \right]\\
&=&
\sqrt{2}\left[\Diag\left((e_ie_j^T + e_je_i^T)e\right)
-(e_ie_j^T + e_je_i^T)\right]
\\&=& \sqrt{2}\left[\Diag\left(e_i+e_j\right)-e_ie_j^T-e_je_i^T\right]\\
&=& \sqrt{2}\left(e_ie_i^T + e_je_j^T - e_ie_j^T - e_je_i^T\right).
\end{array}
\]
Plugging this into \cref{eq:eqnsfindRT}, we obtain
\begin{equation}
\label{eq:adjointrepexp}
\begin{array}{rcl}
\sqrt{2}\svec \left(v_i^Tv_i + v_j^Tv_j - v_i^Tv_j - v_j^Tv_i\right)
&=&
\sqrt{2}\svec \left(v_i^T(v_i-v_j) - v_j^T(v_i - v_j)\right)
\\&=&
\sqrt{2}\svec \left((v_i-v_j)^T(v_i - v_j)\right)
\end{array}
\end{equation}
as a column of the matrix representation of the adjoint. 
Here $v_i$ is the $i$-th row vector of $V$.
The advantage here is that we can use unit vectors for $g$ restricted
to the indices corresponding to $I_1\cup I_2$ when finding the columns.
\end{enumerate}

In our implementations, the calculation of the matrix representation was
the most expensive step. We now show how to avoid this step.

This provides a simplification for solving the least squares problem
in~\cref{eq:lssR}.
\begin{cor}
Let 
\[
D\in \En, \,G = \cK^\dagger(D), 
\]
and let $\alpha \subset [n], |I| \geq d+1$. Define
\[
D_\alpha := \cP_\alpha(D), \,  G_\alpha:=\cK^\dagger(D_\alpha).
\]
Suppose that $\rank(G_\alpha) = \rank(G) = d$. Then
\[
\begin{array}{rcl}
\cK^\dagger \cP_\alpha \cK(G)
&=&
\cK^\dagger  \cK(\cP_\alpha(G))
\\&=&
J(\cP_\alpha(G))J.
\end{array}
\]
Moreover, if $G=WW^T, G_\alpha=W_\alpha W_\alpha^T$ are full rank factorizations, and
$V,V_\alpha$ are centered facial vectors with 
\[
\range(V) = \range(G), \, \range(V_\alpha) = \range(G_\alpha), 
\]
then
\[
Q = V^\dagger W  = (JV_\alpha)^\dagger W_\alpha, \quad W = W_\alpha,
\]
and the \EDM can be recovered with $D=\cK(VQQ^TV^T)$.
\end{cor}
\begin{proof}
	Let $G\in\Snp\cap \SC, face(G) =V\Sc^d_+ V^T \unlhd \Snp$, 
	$V$ a given centered, $V^Te=0$, facial vector of full rank. The first part of the corollary is a direct consequence from~\Cref{lem:CommutePK}.
	Let $G=WW^T$ be a full rank factorization and let $Q$ be a solution of
	\[
	J\cP_\alpha V Q = W_\alpha.
	\]
	Then,
	\[
	\begin{array}{rcl}
		\cK^\dagger\cP_\alpha\cK(VQQ^TV^T) &=& J(\cP_\alpha(VQQ^TV^T))J\\
		&=& JV_\alpha QQ^TV_\alpha^T J\\
		&=& JV_\alpha Q(JV_\alpha Q)^T\\
		&=& W_\alpha W_\alpha^T = G_\alpha.
	\end{array}
	\]
	Therefore,
	\[
	\cP_\alpha\cK(VQQ^TV^T) = \cK(G_\alpha) = D_\alpha
	\]
	and thus
	\[
	\cP_\alpha\big(\cK(VQQ^TV^T)-D\big) = 0.
	\]
\end{proof}

Using the above lemma we can avoid finding the matrix representation and
find the full rank factoriazation of the small $R$ instead.

\subsubsection{Case 3: Small Remaining Block}
\label{sect:casesmallblk}
Suppose that there is a single last block $I$ with $G_\alpha$ that is not a proper
Gram matrix with the correct rank but is too small to divide further,
e.g.,~$<2d+2$. Without loss of generality, we assume
$I=\{1,2,\ldots,\ell\}$. Then there are $\ell(\ell-1)/2$ possible
elements that are noisy. We can now use other distances to find the
noisy one, i.e.,~for $i,j\in I, i\neq j$, we set 
\[
I_{ij} =  \{i,j,\ell+1,\ell+2, \ldots \ell+t\}, \, |I_{ij}| = d+1,
\]
and verify whether or not $K^\dagger\left(D_{I_{ij}}\right)$ is a Gram
matrix with rank $d$. As soon as we find the one that is not, then we
have found the noisy position $i,j$. We then choose a \emph{well conditioned}
block $I_0, |I_0|\geq d+1$ and set
\[
I_1 = I_0 \cup \{i\},\, I_2 = I_0 \cup \{j\}.
\]
We then use~\Cref{sect:casecomplFR} to find the true value of $D_{ij}$.

\subsection{Multi-Blocks with Facial Vectors, \FVp, for \FRp, \MBFVp}
\label{sect:multblcks}
The second approach uses overlapping principal submatrices,
matrix blocks, where the
overlap is minimal size $\geq d+1$ but with embedding dimension $d$.
Therefore, we have a larger number of principal submatrices but they are
significantly smaller. If at anytime we 
find a block that is not a \EDMp, then we stop
and find the noise for this small problem by using submatrices with the
correct embedding dimension.
We denote this approach $\MBFV$.
\index{\MBFVp, multiple-blocks with \FVp}
\index{multiple-blocks with \FVp, \MBFVp}

Rather than using exposing vectors as done above in the bisection
approach~\Cref{sect:twoblks},
the \FR is done by efficiently finding the intersection of two
faces corresponding to two blocks by finding the \FV using the approach 
in~\cite[Lemma 2.9]{kriswolk:09}. We find a \FV for the
small blocks using \Cref{thm:FRone}. We then find the new \FV that
represents the intersection of faces, i.e.,~the union of blocks,
using~\Cref{lem:twocliquered}.
If we end up with a small final block, we use the same strategy as
in~\Cref{sect:casesmallblk}.

\begin{remark}
We chose to divide the problem using principal submatrices. But it is
clear that we could permute the columns and rows $D \leftarrow Q
DQ^T$, where $Q\in \Pi$ is a permutation matrix, before applying the
subdivisions and the algorithm. 
In fact, we could use overlapping cliques in the graph as
long as we maintain a chordal structure, as chordalty allows for \EDM
completion, see e.g.,~\cite{MR96g:15031,MR1321802}.
\end{remark}

\subsection{Equivalent Approach Using Gale Transforms}
\label{sect:Gale}

\index{\SBGTp, small block with Gale transform}
\index{small block with Gale transform, \SBGTp}
In this section we  present an alternative approach, using Gale
transforms that is
equivalent to that of \FR and exposing vectors discussed above. 
We use the smallest blocks along with Gale transforms and denote this as
\SBGTp. We note that the notion of Gale transform \cite{MR1976856,Gale:56} 
is well known and widely used in the theory of polytopes.   
Our approach reveals interesting relationships between \FR and Gale
transforms.
We only consider the multi-block case presented in~\Cref{sect:multblcks}.

We recall from~\Cref{sect:backgrProb}
that we have points $p_1, \ldots, p_n$ in $\R^d$ and we assume 
that the affine hull of these points has full dimension $d$. 
(In fact after centering as
in~\cref{eq:translateP}, we can assume they are centered and
span $\R^d$.) Recall that the $n \times d$ matrix of points $P$ with
\index{configuration matrix, $P$}
\index{$P$, configuration matrix}
$P^T= \begin{bmatrix} p_1& \ldots&  p_n \end{bmatrix}$,
is called the \emph{configuration matrix} of these points. Note that $P$ has 
full column rank, and the Gram and \EDM matrices defined by these points are
$G=PP^T$ and $D = {\cal K}(G)$, respectively.  
The \textdef{Gale space of $D$, $\gal(D)$}, is  given by 
\index{$\gal(D)$, Gale space of $D$}
\begin{equation}
\label{eq:galnulGe}
\gal(D) = \nul \left( \left[ \begin{array}{r} P^T \\ e^T \end{array}
\right]\right) 
 = \nul \left( \left[ \begin{array}{r} G \\ e^T \end{array}
\right]\right). 
\end{equation}
Any $n \times (n-d-1)$ matrix $N$ such that the columns of $N$ form a basis of
gal($D$) is called a {\em Gale matrix} of $D$.
The $i$-th row of $N$ is called a Gale transform
of $p_i$. Note that $N$ is not unique. 
In addition, we recall that a face $f\unlhd \Snp$ is characterized by
the nullspace or range space of any $X\in\relint(f)$ and $f=\face(X)$,
i.e.,~$f$ is the \emph{minimal face containing $X$}.
\index{$\face(X)$, minimal face containing $X$}
\index{minimal face containing $X$, $\face(X)$}
\begin{thm}
\label{thm:exposingGaleequiv}
Let $G\in \Snp\cap \SC$ be a centered Gram matrix of the \EDM $D$.
Let $Z$ be a maximum rank centered exposing vector for $G$ 
as given in~\cref{eq:GZeposdef}. 
Equivalently, $Z\in \relint (\face(G)^c\cap \SC)$, where
$\cdot^c$ denotes \textdef{conjugate face}.
Then any full rank factorization $Z=NN^T$ yields a Gale matrix $N$ of $D$.
Conversely, if $N$ is a Gale matrix of $D$, then $Z=NN^T$ is a 
maximum rank centered exposing vector of $G$.
\end{thm}
\begin{proof}
This follows from the definitions in~\cref{eq:GZeposdef,eq:galnulGe}.
\end{proof}

Recall that a set of points $\{p_1, \ldots, p_k\}$ in $\R^d$ is said to 
be \textdef{affinely independent} if
\[
 \nul \left( \left[ \begin{array}{rrr} p_1 & \cdots & p_k \\ 
                          1 &  \cdots & 1 \end{array} \right]\right) = \{ 0 \}. 
\] 
It is easy to see that
the columns of the Gale matrix $N$ encode the affine dependency of the
points $p_1, \ldots, p_n$. The 
following~\Cref{lemgp} is an immediate consequence of this definition. 

\begin{lem} 
\label{lemgp} 
Let  points $p_1, \ldots, p_k  \in \R^d, k>d$, be in 
\textdef{general position}, and let $P$ be
their configuration matrix. Then every (square) submatrix of 
$[ P \;   e ]$  of order $d+1$ is nonsingular. 
\end{lem}
\begin{cor}
Let $P$ be the configuration matrix in \Cref{lemgp} with corresponding
\EDM $D$. Let $\alpha \subset
\{1,\ldots,k\}, |\alpha | >d$, and let $D_{\cI}$ be the corresponding 
principal submatrix of $D$. Then the rank of the centered Gram matrix
\[
\rank(\cK^\dagger(D_{\cI})) = d.
\]
\end{cor}

Given $D\in \En$ of embedding dimension $d$, and $j\in [n-d-1]$,
let $D_j=D_{[j,j+d+1]}$ be the principal submatrix of $D$ induced by the
columns (rows) $[j,j+d+1]$.  Therefore, each $D_j$ is an \EDM of order
$d+2$. Furthermore, for $j\in [n-d-2]$, the submatrices $D_j$ and $D_{j+1}$ 
overlap in $d+1$ columns (rows).   

Now suppose we have an incomplete $D\in \En$ where only the principal
diagonal blocks
$D_i, i\in [n-d-1]$ are known, while the entries of $D$ 
outside these diagonal
submatrices are not known or are noisy.
The problem addressed in this paper is how to recover all the entries of
$D$, i.e.,~how to complete the \EDMp.
We now show how to recover $D$ by computing $P^0$, a configuration 
matrix of $D$, using the Gale matrix $N$. Recall that we have assumed
that the points $p_1, \ldots, p_n$ that generate $D$ 
are in \textdef{general position}.  We now build $N$
using only information from the diagonal submatrices $D_i, i\in [n-d-1]$. 
We compare this to~\cref{eq:GNInullsp} and see the relation between the
Gale matrix and exposing vectors.

Recall that $\cK^\dagger(D_j) = -\frac{1}{2} J D_j J$. Let 
\begin{equation} 
\label{eq1}
G_j= \cK^\dagger(D_j) 
\end{equation}
be the Gram matrix of $D_j$. Then a Gale matrix $N_j$ corresponding to
$D_j$ is the ($d+2$)-vector which forms a basis of\footnote{
See also \cref{eq:GNInullsp} where $N_i$ is not necessarily a single
column as the size of $G_i$ can be larger than $G_j$
in~\cref{eq:galnulGetwo}.}
\begin{equation}
\label{eq:galnulGetwo}
\nul \left( \left[ \begin{array}{c} {G_j} \\ e^T \end{array}
\right]\right).
\end{equation}
Moreover, it follows from \Cref{lemgp} that all the entries of $N_j$ are nonzero. 
Now a Gale matrix $N$ for the entire matrix $D$ can be built one column 
at a time\footnote{This step is equivalent to the summation of exposing
vectors in~\cref{eq:sumWs}. It is shown in~\cite{ChDrWo:14} that the
summation reduces noise if the data is random and 
from a (Gaussian) normal distribution.}
as follows: the entries of the $j$th column of $N$ are zeros except
in the positions $i = j, \ldots j+d+1$ which are equal to those of $N_j$.
Obviously, by construction, $N$ is $n \times (n-d-1)$ of full column rank.

Now let $V$ be the $n \times d$ matrix whose columns form a basis
of\footnote{This is equivalent to the $V$ found in~\cref{eq:GVRVGE}.}
\[
 \nul \left( \left[ \begin{array}{r} N^T \\ e^T \end{array}
\right]\right). 
\]
Note that a configuration matrix of $D$ is given by 
\begin{equation} \label{eq2}
 P^0 = VQ, 
\end{equation} 
for some nonsingular $Q$ of order $d$,
since the columns of both $P^0$ and $V$ are both bases of the same space.
Once we find $Q$,
both $P^0$ and consequently $D={\cal K}(P^0 {P^0}^T)$ can be recovered. 
Note that (\ref{eq2}) is an overdetermined system. We see
next how to find $Q$ by considering only the first $d+2$ equations of
\cref{eq2} that we denote by:
\begin{equation} \label{eq:P0V1Q} 
  P^0_1 = V_1 Q.
\end{equation}
Here both $P^0_1$ and $V_1$ are  $(d+2) \times d$.\footnote{
Finding $Q$ is equivalent to solving the
system for $r$ in~\cref{eq:eqnsfindR}. In fact, we have the equivalence
$R=QQ^T, G = VQQ^TV^T = VRV^T$. Using the smaller system with $V_1$
would be equivalent to not using the entire overdetermined system 
in~\cref{eq:eqnsfindR}.
}

Now \cref{eq:P0V1Q} cannot be solved as is since both $P^0_1$ and $Q$ are
unknown. In order to overcome this hurdle, we multiply both sides
of \cref{eq:P0V1Q} from the left with $J$ and let $V^0_1:= J P^0_1$.
Thus we get   
\begin{equation}
\label{eq4} 
V^0_1 =  J P^0_1 = J V_1 Q.
\end{equation}
Note that $V^0_1$ is a configuration matrix corresponding to $D_1$ and 
$(V^0_1)^T e_{d+2} = 0$. Recall that $D_1$ is the principal submatrix of $D$ induced by
the rows (columns): $1,\ldots, d+2$.
Now  let $B_1$ be the Gram matrix corresponding to $D_1$. Then 
$V^0_1$ be can be found by the full-rank factorization of $B_1$, i.e., 
$B_1 = V^0_1 {V^0_1}^T$. Thus, Equation \cref{eq4}
can be solved since $Q$ is the only unknown. Then \Cref{lem:equivJXsys}
shows that \cref{eq:P0V1Q,eq4} are equivalent.  

.
\begin{lem}
\label{lem:equivJXsys}
Let $V_1$, $P^0_1$ and $V^0_1$ be as defined above. 
Under the general position assumption,
consider the two systems in the $d \times d$ variable matrix $X$  
\[ 
\mbox{System I:}  \;\; V_1 X = P^0_1 \;\;\;\;\;\;\;\;\;\;\; \mbox{System II: }
                           \;\;  J V_1 X = J P^0_1 = V^0_1.
\]
Then both systems have the same unique solution.
\end{lem}

\begin{proof}
First note that, by relabeling the points if necessary, we can assume 
without loss of generality that $P^0_1$ has full column rank. Moreover,
it follows from \cref{eq:P0V1Q} that rank $(V_1)$ =  rank $(P^0_1)$ and  
rank $([V_1 \; e])$ =  rank $([P^0_1 \; e])$ since   
$[P^0_1 \; e]$ =  
$[V_1 \; e] \left[ \begin{array}{rr} Q & 0 \\  0 &  1 \end{array} \right]$. 
Thus, $V_1$ has full column rank.

Now $Q$ is the unique solution of System I.
It is easy to see that the solution of System II is $Q + Y$ where the columns
of $Y$ are in null$(J V_1)$. Let $y$ be a nonzero vector in null$(J V_1)$, then 
$V_1 y = \alpha e$ for some scalar $\alpha$. Thus 
rank $([V_1 \;\; e])$ = rank $([P^0_1 \;\; e]) = d$
which contradicts \Cref{lemgp} since we assume that $p^1, \ldots, p^n$
are in general position. 
Thus null$(JV_1)$ is trivial and the result follows.
\end{proof}

An immediate consequence of \Cref{lem:equivJXsys} is that $Q$ in 
(\ref{eq2}) can be 
calculated from System II, for example $Q = ({V_1}^T J V_1)^{-1} {V_1}^T V^0_1$.  
Consequently, $P^0 = V Q$ and $D$ = ${\cal K}(P^0{P^0}^T)$.

\subsubsection{Example Using Gale Transforms}
\begin{example}
Consider the following \EDM  $D$ of embedding dimension $d=2$,
\[ D= \left[ \begin{array}{rrrrrr} 0 & 2 & 5 & 9 & 5 & 2 \\ 
  & 0 & 1 & 5 & 5 & 4 \\  &  & 0 & 2 & 4 & 5 \\  &  &  & 0 & 2 & 5 \\
 & & & & 0 & 1 \\ & & & & &  0   \end{array} \right].
\]
(For both the \EDM and Gram matrices we only provide the upper triangular
parts.)
And assume that 
noise is added to the entries $d_{15}$, $d_{16}$ and $d_{26}$. 
Then the Gram matrix corresponding to $D_1$ is  
$G_1= \frac{1}{2}
   \left[ \begin{array}{rrrr} 5 & 1 & -2 & -4\\  & 1 & 0 & -2 \\
 &  & 1 & 1 \\  &  &  & 5 \end{array} \right]$. 
Therefore, $V^0_1$, a configuration matrix of $D_1$, is obtained by the full-rank
factorization of $G_1$, i.e., 
$V^0_1=  \frac{1}{2} \left[ \begin{array}{rr} 1 & -3\\ -1  & -1 \\ 
           -1 & 1 \\ 1 & 3 \end{array} \right]$. 
Furthermore,  a Gale matrix for $D$ is
$N=  \left[ \begin{array}{rrr} -1 & 0 & 0 \\ 3 & -2 & 0 \\ -3 & 3 & -1 \\ 
                     1 & -2 & 2 \\ 0 & 1 & -3 \\ 0 & 0 & 2 \end{array} \right]$. 
Hence, $V$, the matrix whose columns form a basis of  
 $\nul\left( \left[ \begin{array}{r} N^T \\ e^T \end{array}
\right]\right)$, is  
$V=  \left[ \begin{array}{rr} 3 & 0 \\  0 & 1 \\ -2 & 1 \\ -3 & 0 \\
      0 & -1 \\ 2 & -1  \end{array} \right]$ and
hence  
$V_1=  \left[ \begin{array}{rr} 3 & 0 \\  0 & 1 \\ -2 & 1 \\ -3 & 0 \\ \end{array} \right]$ 
and
$JV_1=  \frac{1}{2} \left[ \begin{array}{rr} 7 & -1 \\  1 & 1 \\ 
  -3 & 1 \\ -5 & -1  \end{array} \right]$. 
Hence, solving \cref{eq4}  $V^0_1 = JV_1 Q$, we get 
$Q=  \frac{1}{2} \left[ \begin{array}{rr}  0 & -1 \\ -2 & -1 \\ \end{array} \right]$ 
Therefore, 
$P^0 = V Q =  \frac{1}{2} \left[ \begin{array}{rr} 0 &  -3 \\ -2 & -1 \\ 
-2 & 1  \\  0 & 3 \\ 2 & 1 \\ 2 & -1  \end{array} \right]$
and the full \EDM $D$ is recovered by using $D = {\cal K}(P^0{P^0}^T)$.  
\end{example}

\section{Empirics and Complexity}
\label{sect:empricseff}

We generate random problems based on: the number of points
$n$; the embedding dimension $d$; the magnitude of the noise.\footnote{
We used MATLAB version R2022b on the two servers at University of
Waterloo:
biglinux,  cpu149.math.private 	Dell PowerEdge R840 	
four Intel Xeon Gold 6230 20-core 2.1 GHz (Cascade Lake) 	768 GB;
and
fastlinux,  cpu157.math.private 	Dell PowerEdge R660 	
Two Intel Xeon Gold 6434 8-core 3.7 GHz (Sapphire Rapids) 	256 GB
}
Our table compares three methods: the first uses two blocks and \FR with
exposing vectors; the second uses a minimum number
of multiblocks with facial vectors;
and the third uses Gale transforms with the maximum number of
multiblocks. Our output indicates that all the problems were solved
successfully and this follows from the fact that the general position 
property holds generically. The output includes: the
relative error for the accurate \EDM found; and the
time in cpu seconds.  We discuss the hard cases below
in~\Cref{sect:hardcases}.
	
\subsection{Random Problems}
\Cref{table:fastlinuxMar31,table:biglinuxthree,table:biglinux60to120K}
(on pages
\pageref{table:fastlinuxMar31}, \pageref{table:biglinuxthree},
\pageref{table:biglinux60to120K})
illustrate the high efficiency of the
algorithms for speed, accuracy, and size.
The noise $\alpha$ was a normal random variable with nonzero absolute value
greater than $.01$. Both the position and a near machine precision
accurate value for the noise was found in $100\%$ of the instances.

%

\begin{table}[h]
\tiny
\include{XtableMar31hugeforFastLinux1to30K}
\caption{
Fastlinux; $n$ = 1K to 30K; mean of $3$ instances per row}
\label{table:fastlinuxMar31}
\end{table}



\begin{table}[h]
\include{TableHugeLinux}
\caption{
BigLinux; Multi-block solver with gen time; mean of $3$ instances per row}
\label{table:biglinuxthree}
\end{table}

\begin{table}[h]
\include{XtableBigLinux60to120KApr4}
\caption{
BigLinux; Multi-block solver with gen time; mean of $3$ instances per row}
\label{table:biglinux60to120K}
\end{table}

\subsection{Complexity Estimates}
\label{sect:estimateswork}

We now look at theoretical complexity estimate results and compare them to
the empirical output. For randomly generated problems,
we have plotted dimension versus solution time in 
\Cref{fig:dimtimesemilog}, page \pageref{fig:dimtimesemilog}.
This agrees with the estimates
in~\cref{eq:TwoBruntime,eq:MBruntime,eq:Galeruntime} for the three methods,
respectively:
 \[
\BIEVp,   \MBFVp, \SBGTp: \quad O(n^3),  O(n),  O(n^3).
\]
Note that the expense for generating the random problems is $O(n^2)$ as
the main work is the multiplication using the configuration matrix $PP^T$.

\begin{figure}[h]
\vspace{-2in}
\begin{center}
\hspace*{-3cm}
\includegraphics[height=7in,width=7in]{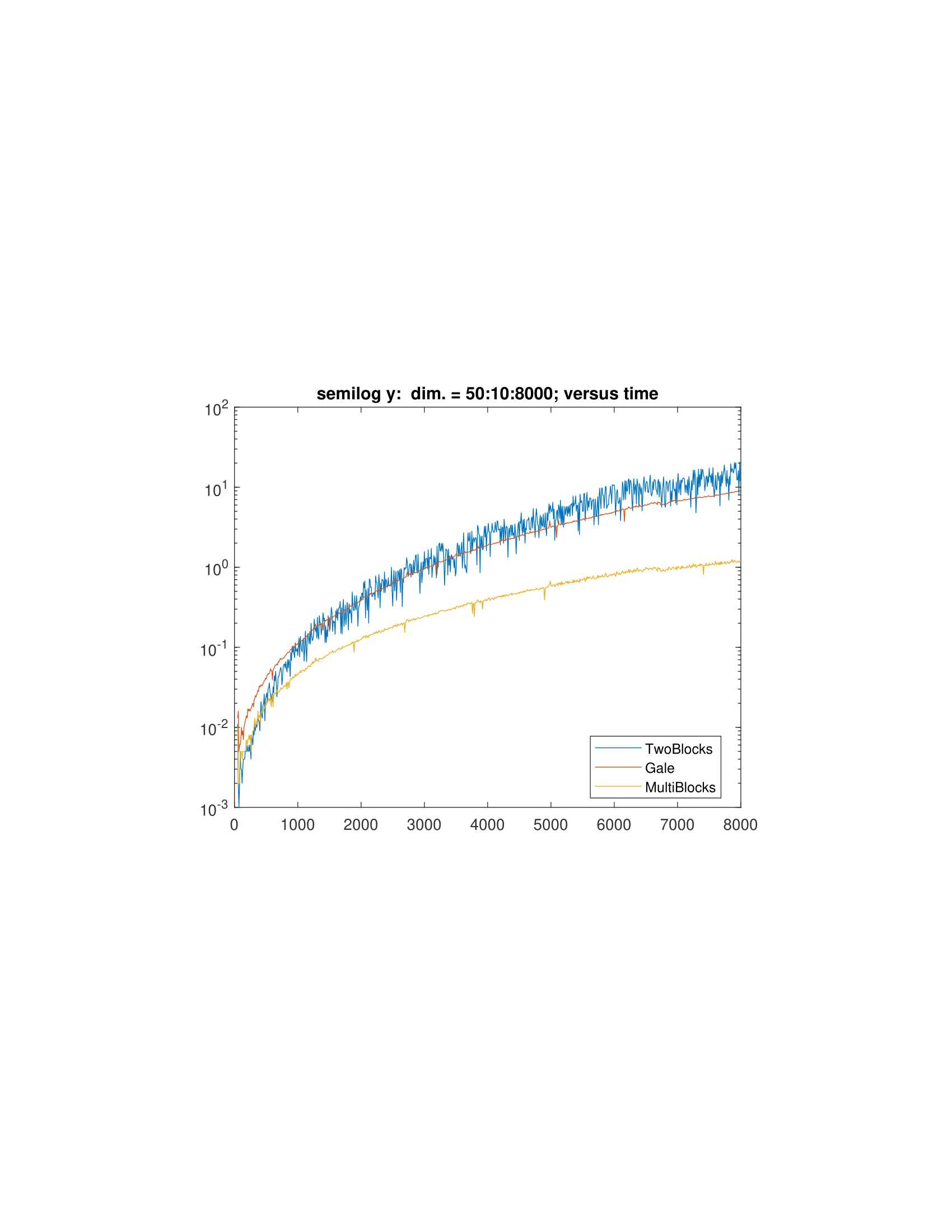}
\vspace{-2.5in}
\caption{Semilogy plot; dimension versus solution time}
\label{fig:dimtimesemilog}
\end{center}
\end{figure}

\subsubsection{Bisection with Exposing Vectors, \BIEVp}
At each step, we divide an $n \times n$ matrix into two blocks with size $\frac{n+d+2}{2}$ so they overlap in a block of size $d+2$. 
The most time consuming calculation in each iteration is calculating the
spectral decomposition of the corresponding Gram matrices
$G_1, G_2$ with runtime $2
O\left(\left( \frac{n+d+2}{2}\right)^3\right) = O(n^3)$, as $n \gg d$. 
The number of points outside of the two principal submatrices is $2
\left( \frac{n-d-2}{2}\right)^2 = \frac{(n-d-2)^2}{2}$. If the noisy
element is outside of the two principal submatrices, the algorithm will
terminate in this step. Thus, the probability to continue to another
step is 
\[
 1 - \frac{(n-d-2)^2}{2n^2} \approx \frac{1}{2}.   
\]
If the noisy element is in one of the principal submatrices, then we
continue with the divide and conquer algorithm. After the division, the
smaller matrix has size $\frac{n+d+2}{2}$, i.e.,~we reduce the size of problem
by approximately half. After $i$ divisions, the size of the matrix is
$\frac{n + 2^{i-1}(d+2)}{2^i}$.
The total runtime of this algorithm is $O(n^3)$. We now drop the $d$ 
and constants in the analysis as $n \gg d$.
Let $T(n)$ be the total runtime of the algorithm, and $f(n)$ be the
runtime of one iteration of the subproblem. Then, we have the recurrence
\[
T(n) = \frac{1}{2} T\left(\frac{n}{2}\right) + f(n) = \frac{1}{2}
T\left(\frac{n}{2}\right) + O(n^3). 
\]
The $O(n^3)$ term dominates the recursive relationship. We get:
\begin{equation}
\label{eq:TwoBruntime}
\text{total runtime is  } O(n^3).
\end{equation}

\subsubsection{Multi-Block with Facial Vectors, \MBFVp}
We build up the full facial vector from
overlapping small principal submatrices of size $2d+6$. Each overlaps
with the previous one with a block of size $d+3$. Solving for
the final connecting matrix $Q$
also involves solving a small system of equations using the first
block, thus costing a constant time. Finding the facial vector of each
small matrix has constant runtime as well, since $n$ is large. The number of
small blocks we look at is roughly $\frac{n}{d+2}$. Therefore,
we get the
\begin{equation}
\label{eq:MBruntime}
\text{total runtime is  } O(n).
\end{equation}

\subsubsection{Small Blocks with Gale Transform, \SBGTp}
We calculate the Gale matrix of each
small principal submatrix of size $d+2$, and then build up the Gale matrix
for the entire matrix $D$. The most time consuming step is in finding $V
\in \R^{n \times d}$, the \FV, whose columns form a basis of $\nul\left(
\begin{bmatrix} N^T \\ e^T \end{bmatrix}\right)$. This has runtime
complexity of:
\begin{equation}
\label{eq:Galeruntime}
\text{total runtime is  } O(n^3).
\end{equation}

\section{Hard Cases; No General Position Assumption}
\label{sect:hardcases}
We now consider problems where the general position assumption may not
hold. This can lead to hard cases for our algorithm. The hard cases are
related to problems where all but $2$ points are in a linear manifold of
dimension $d-2$. We present examples as well as empirics for the hard
cases.

Our modified algorithm for the hard case finds a block with 
rank $d$ by checking all
\emph{consecutive} principal diagonal blocks $D_{[i,i+d]}$ of size $d+1$. 
If we can find a block with the correct embedding dimension $d$ from the principal diagonal blocks, we can proceed with the Gale algorithm as before. Otherwise, we calculate the Gale matrix
$$
N_i = \begin{bmatrix} n_i^T \\ \vdots  \\ n_{i+d}^T \end{bmatrix}, \text{ whose columns form a basis for } \Null\left(\begin{bmatrix} G_i \\ e^T
\end{bmatrix}\right), 
$$
for each small block $D_i := D_{[i, i+d]}$. Since the $G_i$ does not
have rank $d$, the points $p_i, \cdots, p_{i+d}$ are not in general
position. We identify the indices where  the corresponding Gale
transform results in a row of zeros.
If $n_j = 0$, then $p_j$ is not in the affine hull of $\{p_i, \cdots,
p_{i+d}\} \setminus \{p_j\}$ \cite[Section 7.2.1]{alfm18}. 
We collect these indices in a set $I$, so $j \in I$ if there is some $i \in \{1, \cdots, n-d\}$ such that $j \in \{i, \cdots, i+d\}$ and $n_j = 0$ in $N_i$. 
Additionally, let $I' = \{1, \cdots, n\} \setminus I$. We repeat the process on $D_{I \cup I'(1:d+1)}$ until we have $d+1$ points that are in general position. 

There are certain cases where we can not solve the problem, so-called
\textdef{hard cases}.
The results in~\Cref{sect:deq2} show that critical to our
conclusions is the difficult case
when all but two points are in a linear manifold of dimension $d-2$.

\begin{definition}[good, bad blocks]
Let $D_n$ be the noisy \EDM as above. Let
$(i,j)$ be the corrupted index.
\begin{enumerate}
\item
By a \textdef{good block} we mean a principal submatrix $({D_n})_{\alpha}$, 
such that the corresponding
 $\cK^\dagger\left((D_n)_{\alpha}\right)$ is 
positive semidefinite with rank $d$.
\item
By a \textdef{bad block} we mean a principal submatrix $({D_n})_{\alpha}$, 
such that the corresponding
 $\cK^\dagger\left((D_n)_{\alpha}\right)$ is either not positive semidefinite or the
rank is greater than $d$.
\item
An \textdef{uncorrupted good
block} means a principal submatrix $({D_n})_{\alpha}$,
such that the corresponding
 $\cK^\dagger \left((D_n)_{\alpha}\right)$ is positive semidefinite with rank $d$,
and $\alpha$ does not contain either corrupted index $i$~or~$j$.  
\end{enumerate}
\end{definition}

If we assume that we can find an \textdef{uncorrupted good block} 
of size $d+1$, then we will not encounter an unsolvable hard case. 
Under the assumption there exists an uncorrupted good block, we will not encounter the unsolvable case. 
After we found a block $D_\alpha$ of size $d+1$ with embedding dimension $d$, we calculate the Gale matrix for  $D_{\alpha \cup \{ i\}}$ for each $i \in \{1, \cdots, n\} \setminus \alpha$. 
If $D_{\alpha}$ contains the corrupted entry, then by the result in ~\Cref{sect:deq2}, 
the only case we cannot recognize $D_{\alpha \cup \{ i\}}$ to be a corrupted \EDM is if all but two points in this set are in a linear manifold of dimension $d-2$. 
However, under the assumption of the data, there exists $d+1$ points: $\{p_{i_1}, \cdots, p_{i_{d+1}}\}$ in general position and is uncorrupted. 
Therefore, it cannot be that
$D_{\alpha \cup \{i_1\}}, \cdots, D_{\alpha \cup\{ i_{d+1}\}}$
all result in the case of all but two points are in a linear manifold of dimension $d-2$. 
Thus, if $D_{\alpha}$ contains the corrupted entry, we will be able to recognize it is not a proper \EDMp.

\subsection{Characterizing Good and Bad Blocks;
$\embdim(D)=2$ and Beyond}
\label{sect:deq2}
We know that an \EDM $D$ with embedding dimension $d$ corresponds to a
Gram matrix $G = \cK^\dagger(D)\succeq 0$ with $\rank(G) = d$. 
The question that arises is whether, when a single element is
perturbed in $D$, can this corrupted block \emph{always} be
verified correctly by checking that there are no longer exactly $d$ nonzero
and positive eigenvalues for the corresponding Gram matrix
$G$? More precisely, does a corrupted block always have a negative
eigenvalue and/or more than $d$ positive eigenvalues.
We show that indeed this can always be detected except when
all but two points of a corresponding configuration matrix $P$ lie in a
linear manifold of dimension $d-2$. In addition, we show how to handle
these cases by extending the notion of \textdef{yielding} 
in~\cite[Section 7.2]{alfm18}. Note that the $2$ in the
results arises from the fact
that the perturbation matrix $\epsilon E_{ij}$ is rank $2$.

The case when the embedding dimension $d=2$ is special. We conclude that no
failures in detection can occur if the distances are positive, an easy
check. For this special case $d=2$, we do not have to assume that
the problem is in general position to recognize a corrupt entry. 

In our algorithms we have to identify whether the
corrupted element is within a principal submatrix $D_\alpha$, where
for \MBFV the cardinality $|\alpha|\geq d+2$ and is 
often relatively small $|\alpha|<<n$.
We do not assume general position, see~\Cref{lem:genpos}. 
We present the characterization 
in the general $d$ case in~\Cref{thm:genposzerosgenerald}.
This is related to the results on  \textdef{yielding}
in~\cite[Section 7.2]{alfm18},~\cite{MR3842577}. 
Also, this is related to the general question of when a matrix pencil 
is positive semidefinite. The interval for this
is studied recently in~\cite{NGUYEN2024371}.
Our case is special in that we look at rank two updates
\[
D(\epsilon):=D+\epsilon E_{ij}\in\En, D\in  \En,\, 
G(\epsilon) := G + \epsilon \cK^\dagger
(E_{ij})\succeq 0, G\succeq 0.
\]
The specific case of a rank two perturbation as in our case is studied
in~\cite{MR87f:15014}. 
However, we have the additional condition that the rank is maintained at $d$.

Essentially we characterize the cases when
the perturbation $D(\epsilon)$ results in a corresponding
Gram matrix with the correct rank, thus fooling the algorithm.
We use the MATLAB notation \textdef{$\blkdiag$} to denote the 
block diagonal matrix formed from the arguments. 

\begin{theorem}
	\label{thm:genposzerosgenerald}
	Let $D\in \Ek$ with $\embdim(D) = d \leq k-2$.
	Suppose that there is a \underline{single, nonzero}
	corrupted distance $D_{ij}, i<j$, and the noisy \EDM is denoted by the
	\textdef{singular matrix pencil}
	\[
	D(\epsilon) := D + \epsilon E_{ij}, \, \epsilon\in \R\setminus\{0\}.
	\]
	Let $G := \cK^\dagger(D) = Q\Lambda Q^T$ be the Gram matrix with
	its spectral decomposition and, without loss of generality,
	\[
	\Lambda = \blkdiag(0,\Lambda_+),\, \Lambda_+\in \Sc_{++}^d.
	\]
	Denote $G_E:= \cK^\dagger(E_{ij}) = -\frac 12 JE_{ij}J$ and let $\bar G_E$ be defined and appropriately blocked with size $d$:
	\[
	\bar G_E := Q^TG_EQ = 
	\begin{bmatrix} \bar G_{11}&\bar G_{12}\cr \bar G_{12}^T&\bar G_{22},
	\end{bmatrix}, \, \bar G_{22}\in \Sc^{d}.
	\]
	Define the open interval $I_\epsilon$ for maintaining
	$\Lambda_{+}+\epsilon \bar G_{22} \succ 0$, equivalently for 
	maintaining $I_d + \epsilon\Lambda_{+}^{-1/2} 
	\bar G_{22}\Lambda_{+}^{-1/2} \succ 0$, as
	\begin{equation}
		\label{eq:Ieps}
		I_\epsilon = \left(-\frac 1{\lambda_{\max}\left(\Lambda_+^{-1/2}\bar
			G_{22}\Lambda_+^{-1/2}\right)} ,
		-\frac 1{\lambda_{\min}\left(\Lambda_+^{-1/2}\bar
			G_{22}\Lambda_+^{-1/2}\right)}\right)
		\subseteq (-\infty,+\infty),
	\end{equation}
	where a $0$ in the denominator results in $-\infty,+\infty$, for the
	left/right bound of the interval, appropriately. Define the condition:
	\begin{equation}
		\label{eq:EDMcond}
		\textdef{\EDM condition:} \quad
		\fbox{$\epsilon \in I_\epsilon \text{  and  }
			\epsilon \left(\bar G_{11} -  \epsilon \bar G_{12}
			\left(\Lambda_+ + \epsilon \bar G_{22}\right)^{-1} \bar G_{21}
			\right) \succeq 0$}.
\end{equation}
Then:
\begin{enumerate}
\item
The \EDM condition~\cref{eq:EDMcond} defines the convex
yielding interval for $D_{ij}$. More precisely, suppose that 
the \text{\EDM condition } \cref{eq:EDMcond} holds. Then:
\begin{subequations}
\begin{align}
	&      
D(\epsilon)\in \Ek \text{ and } \embdim(D(\epsilon)) \geq d;
				\label{eq:EDMcondthm}
	\\ &    
	\bar G_{11} = 0  \implies  D(\epsilon) \in \Ek, \, \embdim(D(\epsilon))=d;
				\label{eq:EDMcondthmedimd}
\\ &   
	\bar G_{11} -  \epsilon \bar G_{12}\left(\Lambda_+ + 
				\epsilon \bar G_{22}\right)^{-1}
				\bar G_{21}=0 \nonumber
				\\ & \qquad \qquad	\implies  D(\epsilon) \in \Ek, \, \embdim(D(\epsilon))=d.
				\label{eq:EDMcondthmedimdtwo}
			\end{align}
			\label{eq:EDMcondguardall}
		\end{subequations}
		
\item
\label{item:case2yieldingthm}
Conversely, we have the following necessary conditions
for \textdef{restricted yielding}.
		Suppose that there exists $\delta\in\R_{++}$
		such that
		\[
		D(\epsilon) \in \Ek, \, \embdim(D(\epsilon)) = d, \, \forall \epsilon
		\in (-\delta,\delta).
		\] 
		Then the configuration matrix $P$ of $D$, $PP^T = G = \cK^\dagger(D)$,
		has  $k-2\geq
		d$ points that are in a linear manifold $\cL$ of dimension $d-2$.
		Necessarily, the two points $i,j$ outside the linear manifold define the
		corrupted distance $D_{ij}$.
	\end{enumerate}
\end{theorem}
\begin{proof}
	We define the \textdef{matrix pencil}
	\[
	G(\epsilon) := \cK^\dagger(D(\epsilon))  = G + \epsilon G_E.
	\]
	Our results depend on identifying when the perturbed Gram matrix, the
	matrix pencil,
	maintains: $G(\epsilon)=G + \epsilon G_E\succeq 0$ with 
	$\rank G(\epsilon) = d$. We use the spectral decompositions and
	the \textdef{Sylvester law of inertia}. The latter identifies when
	positive semidefiniteness and rank $d$ are maintained 
	under a congruence.
	
	To begin, we need to consider the eigenpairs for the two nonzero eigenvalues of 
	\[
	G_E:= \cK^\dagger(E_{ij}) = -\frac 12 JE_{ij}J.
	\]
	An orthogonal pair of eigenvectors of $E_{ij}$ is $e_i\pm e_j$. We have
	the properties $E_{ij}\neq 0, \trace(E_{ij})=0$ and
	\[
	\lambda_1(E_{ij} ) =1 > 0 = 
	\lambda_2(E_{ij} ) = \ldots =
	\lambda_{n-1}(E_{ij} )=0 > \lambda_n(E_{ij} ) =-1.
	\]
	Therefore, the above singular congruence $JE_{ij}J=\cP_{\SC}(E_{ij})$,
	by~\cref{eq:projcentered}, implies
	\[
	\lambda_1(G_E )  > 0 = 
	\lambda_2(G_E ) = \ldots =
	\lambda_{n-1}(G_E )=0 > \lambda_n(G_E),
	\]
	i.e.,~there are exactly two nonzero eigenvalues of $G_E$;
verified as well in the following.
	Define the two orthogonal vectors
	\begin{equation}
		\label{eq:eigvecGE}
		e_{ij} := e_i-e_j, \,
		e_{ijc} := \frac 2k e - (e_i+e_j) \in e^\perp \subset \Rk.
	\end{equation}
	One can verify that $e_{ij}$ and $e_{ijc}$ are eigenvectors of $G_E$ and the corresponding eigenvalues are:
	\begin{equation}
		\label{eq:eigvalGE}
		\lambda_1(G_E) = \frac 12 > 0 > \lambda_k(G_E) =\frac {2-k}{2k}.
	\end{equation}
	Notice that signs of the eigenvalues are strictly positive and strictly negative. The congruences with $Q$ and the interlace theorem for eigenvalues yield $\lambda_{\max}(\bar G_{22})\ge0\ge\lambda_{\min}(\bar G_{22})$. The congruence with $\Lambda_{+}^{-1/2}$ and the definition \cref{eq:Ieps} of $I_\epsilon$ bring us to $0\in I_\epsilon$ and we conclude that
	\[
	\begin{array}{rcrcl}
		\epsilon\in I_\epsilon &\iff& I_d + \epsilon\Lambda_{+}^{-1/2} \bar G_{22}\Lambda_{+}^{-1/2} &\succ& 0\\
		&\iff& \Lambda_{+}+\epsilon \bar G_{22} &\succ& 0,
	\end{array}
	\]
	and that
	\[
	G+\epsilon G_E = Q(\Lambda+\epsilon\bar G_E)Q^T\succeq0 \iff \Lambda + \epsilon\bar G_E\succeq0.
	\]
	\begin{enumerate}
		\item
		Assuming \cref{eq:EDMcond}, we have $\Lambda_{+}+\epsilon \bar G_E\succ0$. By the Schur complement theorem,
		\[
		\epsilon \left(\bar G_{11} -  \epsilon \bar G_{12}
		\left(\Lambda_+ + \epsilon \bar G_{22}\right)^{-1} \bar G_{21}
		\right) \succeq 0 \iff
		G(\epsilon) \succeq 0.
		\]
		Under the assumptions and definitions on $D(\epsilon), I_\epsilon$,
		and by continuity of eigenvalues and of the 
		linear transformation $\cK^\dagger$, we get that the $d$ positive
		eigenvalues of $G$ are perturbed but stay positive in $G(\epsilon)$, 
		i.e.,~recalling the assumption that $\embdim(D) = d$, we have
		\[
		\rank(G(\epsilon)) \geq \rank(G)=d\implies\embdim(D(\epsilon))\geq d.
		\]
		This proves \cref{eq:EDMcondthm}.
		
		\cref{eq:EDMcondthmedimd} and \cref{eq:EDMcondthmedimdtwo} follow from a
		Schur complement argument. That is,
		\[
		\begin{array}{rcl}
			\Lambda + \epsilon \bar G_E &=& \begin{bmatrix}
				\epsilon \bar G_{11}&\epsilon \bar G_{12}\cr
				\epsilon \bar G_{12}^T&\Lambda_{+} + \epsilon \bar G_{22}
			\end{bmatrix}  \\
			&\cong&
			\begin{bmatrix}
				\epsilon \left(\bar G_{11} - \epsilon \bar G_{12}
				\left(\Lambda_+ + \epsilon \bar G_{22}\right)^{-1} \bar G_{21}
				\right)&\epsilon \bar G_{12}\cr
				0& \Lambda_{+} + \epsilon \bar G_{22}
			\end{bmatrix}.
		\end{array}
		\]
		Since $\rank(\Lambda_{+} + \epsilon \bar G_{22})=d$, $\bar G_{11} - \epsilon \bar G_{12}
		\left(\Lambda_+ + \epsilon \bar G_{22}\right)^{-1} \bar G_{21}=0$ implies
		\[
		\rank(\Lambda+\epsilon\bar G_E)=\rank(G+\epsilon G_E)=d.
		\]
		Similarly, $\bar G_{11}=0$ leads to $\bar G_{12}=0$ and thus $\rank(G+\epsilon G_E)=d$.
		Therefore~\cref{eq:EDMcondguardall} holds.
		
		\item
		Suppose that there exists $\delta >0$ such that, for all
		$|\epsilon|<\delta$, we have $D(\epsilon)\in\Ek$ and $\embdim(D(\epsilon))=d$. This is equivalent to
		\[
		G+\epsilon G_E\succeq0 \text{ and } \rank(G+\epsilon G_E)=d.
		\]
		
		We want to show that
		\[
		G+\epsilon G_E\succeq 0, \,
		\rank(G+\epsilon G_E) = \rank(G)  \implies 
		\range(G_E) \subset \range(G).
		\]
		Recall that $G+\epsilon G_E\succeq0\iff \Lambda+\epsilon\bar G_E\succeq0$ and that
		\[
		\Lambda + \epsilon \bar G_E \cong
		\begin{bmatrix}
			\epsilon \left(\bar G_{11} - \epsilon \bar G_{12}
			\left(\Lambda_+ + \epsilon \bar G_{22}\right)^{-1} \bar G_{21}
			\right)&\epsilon \bar G_{12}\cr
			0& \Lambda_{+} + \epsilon \bar G_{22}
		\end{bmatrix}.
		\]
		Since $0\in I_\epsilon$, we can find small
		enough $|\epsilon|>0$, namely $\epsilon\in I_\epsilon\cap(-\delta,\delta)$, so that $\Lambda_{+}+\epsilon\bar G_{22}\succ0$. Hence $\rank(\Lambda_{+}+\epsilon\bar G_{22})=d$ and note that the rank of the entire matrix is the sum of the ranks of the diagonal blocks. Hence, the top left block after the elimination has to have rank $0$, by $\rank(G+\epsilon G_E)=d$ assumption. Therefore,
		\begin{equation}\label{eq:schuriszero}
			\bar G_{11} - \epsilon \bar G_{12}
			\left(\Lambda_+ + \epsilon \bar G_{22}\right)^{-1} \bar G_{21} = 0.
		\end{equation}
		\cref{eq:schuriszero} is violated for small enough $|\epsilon|$ unless $\bar G_{11}=0$. Hence $\bar G_{11}=0$, which also implies $\bar G_{12}=0$. Hence,
		\[
		\range(\bar G_E)\subseteq \range(\blkdiag(0, I_d)) = \range(\blkdiag(0, \Lambda_{+})) =  \range(\Lambda),
		\]
		which is equivalent to
		\[
		\range(G_E) = \range(Q\bar G_EQ^T)\subseteq \range(Q\Lambda Q^T) = \range(G).
		\]
		
		This implies that the facial vector $V$ for $G=\cK^\dagger(D)$ must have
		range that contains the span of the two eigenvectors of
		$G_E= \cK^\dagger(E_{ij})$. Otherwise, 
		at least one of the following happens: 
		(i) we lose positive semidefiniteness; (ii) the number of 
		positive eigenvalues increases.  Therefore, we must have
		\begin{equation}
			\label{eq:FVGeegne}
			\range\left(\begin{bmatrix}
				e_{ij} & e_{ijc} 
			\end{bmatrix}\right)
			\subset \range(V), \quad 
			V = \begin{bmatrix} e_{ij} & e_{ijc}  & \bar V
			\end{bmatrix}, \quad
			\begin{bmatrix}
				e_{ij} & e_{ijc} 
			\end{bmatrix}^T  \bar V = 0.
		\end{equation}
		
		Thus, the $\ell$-th row of $V$ is
		\[
		V(\ell,:) = \begin{bmatrix} 0 & \frac2k & \bar V(\ell,:)
		\end{bmatrix}\in \{0\}\times\{2/k\}\times \R^{d-2},
		\]
		for all $\ell\neq i,j$.
		This shows that for $R=I_d$,
		all but two points are in a manifold of dimension $d-2$.
		This means the same is true for general $R\in \Sc^d_{++}$. Without loss of generality, let $i=n-1,j=n$. Then,
		\[
		V= \begin{bmatrix}
			V_1 \cr V_2
		\end{bmatrix}\in\R^{k\times d}, V_1\in\R^{(k-2)\times d}, V_2\in\R^{2\times k}.
		\]
		Note that rows of $V_1$ are in a $d-2$ dimensional manifold if, and only if, $\rank V_1\le d-2$. Since $R\succ0$, $V_1R^{1/2}$ maintains the same rank as $V_1$, i.e., the rows of $V_1R^{1/2}$ are in a manifold of dimension $d-2$ where
		\[
		\begin{array}{rcl}
			G &=& VRV^T = [VR^{1/2}][VR^{1/2}]^T\cr
			&=& \begin{bmatrix}
				V_1R^{1/2}\cr V_2R^{1/2}
			\end{bmatrix}
			\begin{bmatrix}
				V_1R^{1/2}\cr V_2R^{1/2}
			\end{bmatrix}^T.
		\end{array}
		\]
	\end{enumerate}
\end{proof}

\begin{cor}
\label{cor:genposzeros}
Let $D\in \Ek$ with $\embdim(D) = 2 < k-1$.
Suppose that there is a corrupted position $i<j$ and the noisy \EDM is
\begin{equation}
\label{eq:epssmallbnddistwo}
D_n = D + \epsilon E_{ij}, \,  \epsilon \neq 0.
\end{equation}
If $D_n\in \Ek, \embdim(D_n)\leq d$, and $|\epsilon|$ is sufficiently
small, then the (centered) configuration
matrix $P$ of $D$, $PP^T = G = \cK^\dagger(D)$, has (at least) $k-2\geq
2$ points that are equal, i.e.,~ the rows $P_{s:} = p\in \Rd, 
\forall s\in [k], s\neq i, s\neq j$. Thus 
\[
D_{st}=(D_n)_{st}=0, \forall s,t\in [k], s\neq i, t\neq j.
\]
\end{cor}
\begin{proof}
We apply \Cref{thm:genposzerosgenerald},~\Cref{item:case2yieldingthm}
for this special $d=2$ case with $|\epsilon|$ small to guarantee we are
within the interval.
\end{proof}

\subsection{Example of Hard Case; Multiple Solutions}
The following is an example where we have determined a \emph{bad} block
but where we cannot determine the \emph{position} of the entry 
that is corrupted. 
Consider the following corrupted \EDM in embedding dimension 3. 
\[ 
\begin{array}{rcl}
D(18) := D_n = \begin{bmatrix}
0 & 4 & 16 & 8 & 6 & 14 \\
4 & 0 & 4 & 4 & 6 & 6 \\
16 & 4 & 0 & 8 & 14 & 6 \\
8 & 4 & 8 & 0 & \fbox{$18$} & 14 \\
6 & 6 & 14 & \fbox{$18$}  & 0 & 20 \\
14 & 6 & 6 & 14 & 20 & 0 
\end{bmatrix} 
\quad
G_n = \cK^\dagger(D_n); \,\,
\lambda(G(18)) = \lambda(G_n) \approx
\begin{pmatrix} 
  -0.7252\\
    0.0000\\
    0.0000\\
    4.7157\\
    7.4003\\
   13.2759
\end{pmatrix}. 
\end{array}
\]
This is not an \EDM, as its corresponding $G_n=\cK^\dagger(D_n)$ is not positive
semidefinite. 
However, there is more than one way to change only one entry of this matrix
and obtain an \EDM. 
We can change the $18$ in the $(4,5)$ entry in two different ways, to
$14$ and $6/5$, respectively, and get
\EDMp s $D_1:=D(14), D_2:=D(6/5)$, with corresponding positive semidefinite
$G(\alpha)=\cK^\dagger(D(\alpha))$: 
\[ 
\lambda(G(14)) \approx 
\begin{pmatrix}
  0.000000   \cr  0.000000 \cr    0.000000 \cr    4.876894  \cr   6.000000
\cr 13.123106
\end{pmatrix},
\qquad
\lambda\left(G\left(\frac 65\right)\right) \approx 
\begin{pmatrix}
  0.000000 \cr    0.000000  \cr   0.000000  \cr   1.562857 \cr
6.189609  \cr   14.114201 
\end{pmatrix}.
\]
However, if we use a value in the middle of the changes $[6/5,14]$,
e.g.,~$D(5)$, then the corresponding $G(5)$ is indeed a Gram matrix but
has $4>d$ positive eigenvalues. This corresponds with 
\Cref{thm:genposzerosgenerald} and means that we have the end points of
a yielding interval where $D(6/5) + \epsilon E_{45}\in \cE^6$ if, and only if,
$\epsilon\in [0,12.8]$. This also means that the two eigenvectors
of $G_E= \cK^\dagger(E_{45})$ are not in $\range(G(6/5)$ or
$\range(G(14))$.
In fact, by checking the rank of $[G~v]$, with $v$ one of the
eigenvectors for $G_E$, we can see that one eigenvector is in the range
while the other is not, thus explaining the finite yielding intervals.
The example can be extended to higher embedding dimension using yelding
intervals formed with the eigenvectors in the appropriate range.

\begin{figure}[h]
\vspace{-2in}
\begin{center}
\includegraphics[height=7in,width=5in]{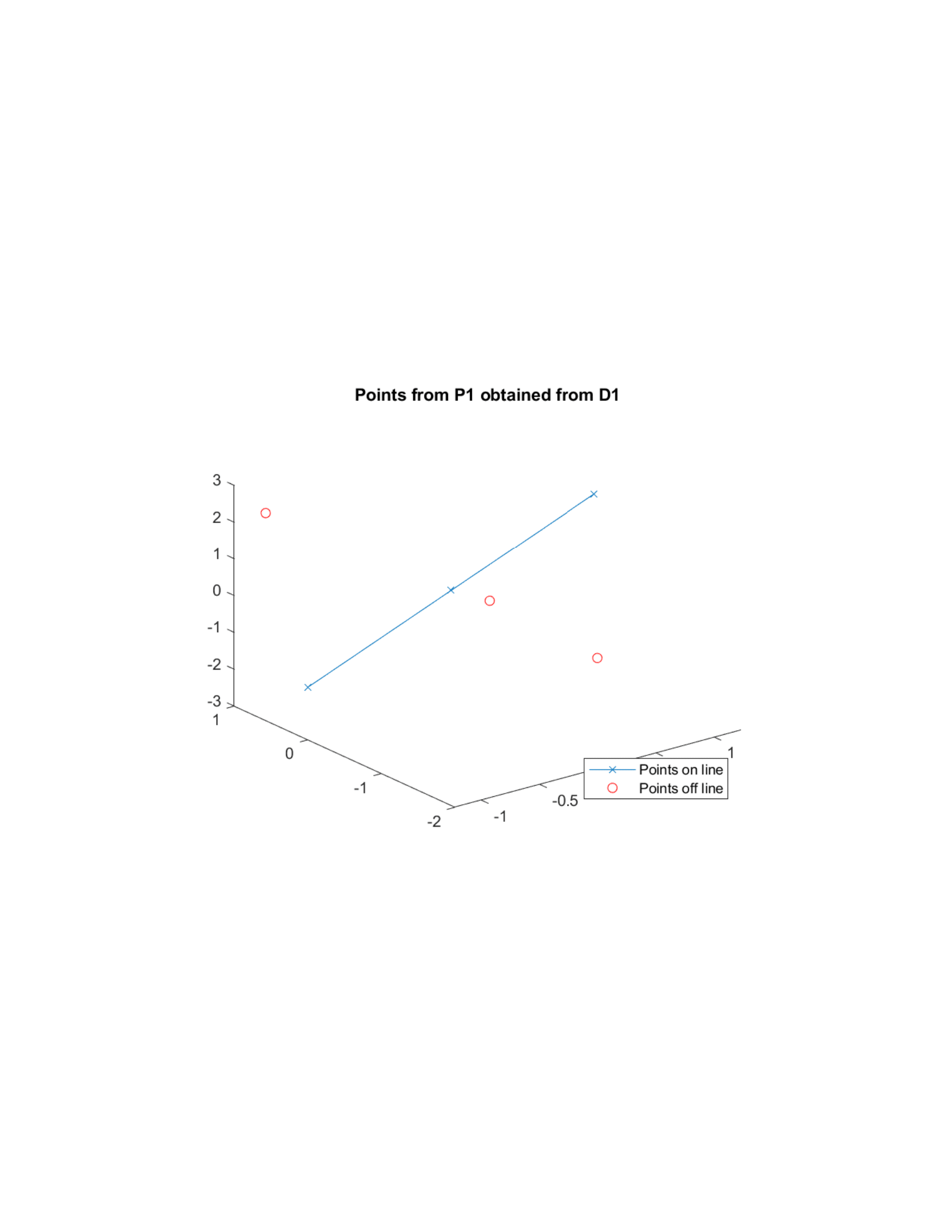}
\vspace{-2.2in}
\caption{Three points off the line}
\label{fig:ptsD1}
\end{center}
\end{figure}
Or we can change the $(4,6)$ entry in 2 different ways and obtain a proper \EDM
\[ D_3 = \begin{bmatrix}
0 & 4 & 16 & 8 & 6 & 14 \\
4 & 0 & 4 & 4 & 6 & 6 \\
16 & 4 & 0 & 8 & 14 & 6 \\
8 & 4 & 8 & 0 & 18 & \fbox{$\frac{42}{5}$} \\
6 & 6 & 14 & 18 & 0 & 20 \\
14 & 6 & 6 & \fbox{$\frac{42}{5}$} & 20 & 0 
\end{bmatrix} \;\;   D_4 = \begin{bmatrix}
0 & 4 & 16 & 8 & 6 & 14 \\
4 & 0 & 4 & 4 & 6 & 6 \\
16 & 4 & 0 & 8 & 14 & 6 \\
8 & 4 & 8 & 0 & 18 & \fbox{$2$} \\
6 & 6 & 14 & 18 & 0 & 20 \\
14 & 6 & 6 & \fbox{$2$}  & 20 & 0 
\end{bmatrix} \]
Similarly, changing the $(5,6)$ entry can also make $D_n$ into an \EDM
\[ D_5 = \begin{bmatrix}
0 & 4 & 16 & 8 & 6 & 14 \\
4 & 0 & 4 & 4 & 6 & 6 \\
16 & 4 & 0 & 8 & 14 & 6 \\
8 & 4 & 8 & 0 & 18 & 14 \\
6 & 6 & 14 & 18 & 0 & \fbox{$14$} \\
14 & 6 & 6 & 14 & \fbox{$14$} & 0 
\end{bmatrix} \;\; D_6 = \begin{bmatrix}
0 & 4 & 16 & 8 & 6 & 14 \\
4 & 0 & 4 & 4 & 6 & 6 \\
16 & 4 & 0 & 8 & 14 & 6 \\
8 & 4 & 8 & 0 & 18 & 14 \\
6 & 6 & 14 & 18 & 0 & \fbox{$6$} \\
14 & 6 & 6 & 14 & \fbox{$6$} & 0 
\end{bmatrix}\] 
This situation occurs in general when all but 3 points are in a manifold of dimension $d-2$. In this case, points 1-3 are on a line in $\R^3$, and point 4-6 are off the line. \\
For example, the configuration of $D_1$, 
see~\Cref{fig:ptsD1}, page~\pageref{fig:ptsD1}, is 
\[ P_1 = \begin{bmatrix}
0 & 0 & 2 \\ 0 & 0 & 0 \\ 0 & 0 & -2 \\ -2 & 0 & 0 \\ 1 & 2 & 1 \\ 1 & -2 & -1
\end{bmatrix} \]
If we perturb the distance between points 4 and 5, then we no
longer have a \EDM.
However, if we ignore the existence of point 6, then all but 2 points are on a manifold of dimension $d-2$ so the $(4,5)$ entry is yielding. Thus,
the submatrix formed by points 1 to 5 is still a proper \EDM. Thus we
can fix the corrupted \EDM by changing the $(4,6)$ entry instead.

\subsection{Empirics for the Hard Case} 
\label{sect:empiricshardcase}
We now generate random hard problems where some points are generated to be on a linear manifold of dimension less than $d$. 
Thus, many points are not in general position. 
Under the Data specification columns, the fourth column indicates the number of points not in general position. 
The fifth column indicate the dimension of manifold these points live in.
We run the problems with $n = 100$ and $d$ from 2 to 10. For each embedding dimension, we run 50 problems.
Our output demonstrate the algorithm works well for the hard problems.

\begin{table}[H]
\tiny
\include{XtableHardProblems}
\caption{
Average of 50 hard problems; hard problems where many point are generated to be on a manifold of dimension less than the embedding dimension $d$}
\label{table:hardproblems}
\end{table}

\section{Conclusion}
\label{sect:concl}
In this paper we have studied a case of error-correction in \EDM with special
structure, i.e.,~we assume that the \EDM $D$ has known embedding
dimension $d$ and that exactly one distance is in error, is corrupted.
We have presented three different strategies for divide and conquer
and three different types of facial reduction. Our approaches
accurately identify and correct exactly one corrupted distance of an \EDMp. 
The numerical tests confirm that we
can solve huge problems to high precision and quickly. In fact, the tests
on random problems with $n=100,000$ for the best method 
take approximately $100$ seconds to solve to machine precision; 
and this confirms our analysis of $O(n)$ cost for the best
of our three algorithms that we tested.
Note that $n=100,000, d=3$ means that $P$ has $3(10^6)$ variables and the
Gram matrix is dense and has order $5(10^{12})$ variables. Attempting to solve 
these problems using \SDP with interior-point or
first order methods would not be reasonable and, even if possible, would
not obtain high accuracy; whereas we obtain near machine precision.

We include a characterization of when a perturbation of a single element
results in a \EDM with unchanged embedding dimension $d$.
This is equvalent to maintaining the difficult  constant rank constraint.
Moreover, we provide a characterization for when the \NEDM problem
solves our problem, i.e,~this happen if, and only if, the original data
element $(D_0)_{ij} = 0$ and the perturbation  $\alpha <0$, a highly
degenerate trivial case as the location $ij$ is identified.

In addition, the algorithm extends to any number of corrupted elements that
are outside the blocks that we choose. We can also work with a chordal
graph and choose overlapping cliques to obtain the principal
submatrices.

{\bf Acknowledgement:} The authors would like to thank Walaa M. Moursi
for many hours of helpful conversations.

\bibliographystyle{siam}

\cleardoublepage
\label{ind:index}
\printindex
\addcontentsline{toc}{section}{Index}

\cleardoublepage
\bibliography{.master,.edm,.psd,.bjorBOOK}
\addcontentsline{toc}{section}{Bibliography}

\end{document}

%% file: XtableMar31hugeforFastLinux1to30K.tex
\begin{adjustbox}{width=\textwidth}
\begin{tabular}{|c|c|c|c|c|c|c|c|c|}  \hline
\multicolumn{3}{|c|}{Data specifications} & \multicolumn{2}{|c|}{\BIEVp} &\multicolumn{2}{|c|}{\MBFVp}&\multicolumn{2}{|c|}{\SBGTp} \\  \hline
$n$ & $d$ & noise & rel-error & time(s) & rel-error & time(s) & rel-error & time(s) \\ \hline
1000 & 5 & 0.289 &  1.86e-12 & 0.200 &8.44e-14 & 0.022 & 3.10e-12 & 0.121 \\ \hline
2000 & 5 & -0.235 & 1.22e-12 & 0.916 &1.28e-13 & 0.044 & 3.84e-13 & 0.364 \\ \hline
3000 & 5 & -0.843 &  2.33e-12 & 1.713 &3.36e-13 & 0.130 & 1.28e-13 & 0.638 \\ \hline
4000 & 5 & 0.570 & 7.01e-13 & 1.842 &3.41e-13 & 0.195 & 2.74e-13 & 1.342 \\ \hline
5000 & 5 & 0.517 & 1.10e-12 & 4.394 &8.80e-14 & 0.292 & 2.07e-13 & 1.847 \\ \hline
6000 & 5 & 0.659 & 1.30e-12 & 6.861 &3.07e-13 & 0.414 & 1.31e-12 & 2.889 \\ \hline
7000 & 5 & 0.200 & 2.16e-12 & 11.759 &7.95e-14 & 0.631 & 2.01e-13 & 3.895 \\ \hline
8000 & 5 & -0.240 & 1.71e-12 & 10.993 &1.12e-13 & 0.768 & 2.01e-13 & 5.338 \\ \hline
9000 & 5 & -0.294 & 4.63e-13 & 18.939 &1.73e-13 & 0.974 & 1.54e-13 & 7.299 \\ \hline
10000 & 5 & 0.197 & 2.01e-12 & 23.177 &1.63e-13 & 1.179 & 5.51e-13 & 9.529 \\ \hline
11000 & 5 & -0.405 & 4.42e-12 & 18.598 &7.43e-14 & 1.383 & 3.15e-13 & 11.282 \\ \hline
12000 & 5 & 0.085 & 4.99e-12 & 20.521 &3.88e-13 & 1.732 & 3.16e-13 & 14.150 \\ \hline
13000 & 5 & 0.311 & 1.42e-12 & 44.017 &3.35e-13 & 2.097 & 2.55e-13 & 17.511 \\ \hline
14000 & 5 & -0.390 & 4.50e-12 & 53.028 &7.69e-14 & 2.201 & 2.14e-11 & 20.961 \\ \hline
15000 & 5 & -0.348 & 5.31e-12 & 54.837 &3.78e-13 & 2.383 & 6.69e-12 & 25.517 \\ \hline
16000 & 5 & -0.294 & 6.14e-12 & 51.610 &1.30e-13 & 2.780 & 6.97e-13 & 29.842 \\ \hline
17000 & 5 & 0.063 & 3.08e-12 & 64.764 &2.12e-13 & 3.176 & 3.78e-13 & 33.774 \\ \hline
18000 & 5 & -0.064 & 1.33e-11 & 92.478 &2.20e-13 & 3.485 & 1.27e-12 & 38.898 \\ \hline
19000 & 5 & 0.001 & 2.81e-11 & 99.526 &3.89e-13 & 3.986 & 2.99e-13 & 43.750 \\ \hline
20000 & 5 & -0.004 & 9.56e-13 & 97.926 &9.13e-13 & 4.229 & 4.21e-13 & 51.621 \\ \hline
21000 & 5 & 0.368 & 1.09e-12 & 130.590 &1.93e-13 & 4.749 & 9.99e-13 & 58.367 \\ \hline
22000 & 5 & 0.035 & 1.21e-11 & 177.391 &7.86e-14 & 5.232 & 1.34e-13 & 66.882 \\ \hline
23000 & 5 & 0.018 & 5.76e-12 & 173.445 &2.44e-13 & 5.786 & 1.06e-12 & 73.720 \\ \hline
24000 & 5 & 1.000 & 2.69e-12 & 160.173 &1.38e-13 & 5.970 & 2.18e-13 & 82.715 \\ \hline
25000 & 5 & 0.139 & 4.08e-12 & 242.229 &3.11e-13 & 6.804 & 3.07e-12 & 91.045 \\ \hline
26000 & 5 & -0.385 & 1.91e-12 & 91.091 &1.28e-13 & 6.946 & 3.23e-13 & 102.973 \\ \hline
27000 & 5 & -0.131 & 7.17e-12 & 173.206 &9.43e-14 & 7.791 & 2.47e-13 & 112.248 \\ \hline
28000 & 5 & 0.022 & 1.18e-11 & 245.353 &3.61e-13 & 8.199 & 2.10e-10 & 124.508 \\ \hline
29000 & 5 & 0.109 & 1.05e-11 & 264.581 &6.11e-13 & 8.728 & 3.85e-12 & 134.517 \\ \hline
30000 & 5 & 0.089 & 2.68e-12 & 299.627 &2.58e-13 & 8.793 & 7.09e-13 & 149.871 \\ \hline
\end{tabular}
\end{adjustbox}

%% file: TableHugeLinux.tex
\begin{center}
	\begin{tabular}{|c|c|c|c|c|c|}  \hline
		\multicolumn{4}{|c|}{Data specifications} &\multicolumn{2}{|c|}{\MBFVp} \\  
		\hline $n$ & $d$ & noise & gen-time & rel-error & time(s) \\ \hline
		60000 & 3 & -0.405 & 82.643 & 7.63e-14 & 54.850 \\ \hline
		65000 & 3 & 0.386 & 94.786 & 1.24e-13 & 64.485 \\ \hline
		70000 & 3 & -0.203 & 111.036 & 3.53e-13 & 75.925 \\ \hline
		75000 & 3 & 0.436 & 128.879 & 5.53e-13 & 88.849 \\ \hline
		80000 & 3 & -0.129 & 156.128 & 2.86e-13 & 105.870 \\ \hline
		85000 & 3 & -0.081 & 190.250 & 3.21e-13 & 123.193 \\ \hline
		90000 & 3 & 1.000 & 213.570 & 4.92e-13 & 134.878 \\ \hline 
	\end{tabular} 
\end{center}

%% file: XtableBigLinux60to120KApr4.tex
\begin{center}
\begin{tabular}{|c|c|c|c|c|c|}  \hline
\multicolumn{4}{|c|}{Data specifications} &\multicolumn{2}{|c|}{\MBFVp} \\  \hline $n$ & $d$ & noise & gen-time & rel-error & time(s) \\ \hline
60000 & 5 & 0.436 & 90.767 & 1.32e-13 & 53.049 \\ \hline
70000 & 5 & 0.026 & 109.262 & 3.37e-13 & 72.006 \\ \hline
80000 & 5 & 0.550 & 155.565 & 3.12e-13 & 93.529 \\ \hline
90000 & 5 & 0.435 & 196.015 & 8.91e-13 & 114.084 \\ \hline
100000 & 5 & 0.420 & 243.331 & 9.21e-13 & 144.727 \\ \hline
110000 & 5 & 0.330 & 315.411 & 3.01e-13 & 196.053 \\ \hline
120000 & 5 & 0.205 & 385.318 & 7.16e-13 & 237.335 \\ \hline
\end{tabular}
\end{center}

%% file: XtableHardProblems.tex
\begin{adjustbox}{width=\textwidth}
\begin{tabular}{|c|c|c|c|c|c|c|c|}  \hline
\multicolumn{5}{|c|}{Data specifications} &\multicolumn{3}{|c|}{Hard Gale Transform with Multi Blocks} \\  \hline
$n$ & $d$ & noise & \# pts not in general position& dim of pts &tol-attained & rel-error & time(s) \\ \hline
100 & 2 & 0.016 & 79 & 1 & 100\% & 1.64e-12 & 0.058 \\ \hline
100 & 3 & 0.071 & 86 & 1 & 100\% & 1.62e-13 & 0.050 \\ \hline
100 & 4 & -0.472 & 82 & 3 & 100\% & 5.08e-13 & 0.034 \\ \hline
100 & 5 & -0.328 & 78 & 3 & 100\% & 6.98e-13 & 0.032 \\ \hline
100 & 6 & 0.185 & 89 & 1 & 100\% & 1.08e-12 & 0.048 \\ \hline
100 & 7 & 0.334 & 91 & 1 & 100\% & 3.09e-12 & 0.045 \\ \hline
100 & 8 & -0.193 & 78 & 2 & 100\% & 3.98e-12 & 0.041 \\ \hline
100 & 9 & 0.394 & 76 & 8 & 100\% & 9.31e-12 & 0.028 \\ \hline
100 & 10 & 0.222 & 83 & 5 & 100\% & 1.91e-11 & 0.047 \\ \hline
\end{tabular}
\end{adjustbox}